\documentclass[10pt,a4paper,twoside]{article}
\usepackage[utf8]{inputenc}
\usepackage{amsmath,amsfonts,amssymb,mathrsfs,latexsym,mathtools,bbm,amsthm}
\usepackage[all,cmtip]{xy}
\def\YEAR{\year}\newcount\VOL\VOL=\YEAR\advance\VOL by-1995
\def\firstpage{1}\def\lastpage{1000}
\def\received{}\def\revised{}
\def\communicated{}

\makeatletter
\def\magnification{\afterassignment\m@g\count@}
\def\m@g{\mag=\count@\hsize6.5truein\vsize8.9truein\dimen\footins8truein}
\makeatother

\font\eightrm=cmr8
\font\caps=cmcsc10                    
\font\Caps=cmcsc10 scaled \magstep1   

\makeatletter
\@specialpagefalse\headheight=8.5pt

\renewcommand{\@evenhead}{%
    \ifnum\thepage>\lastpage\rlap{\thepage}\hfill%
    \else\rlap{\thepage}\slshape\leftmark\hfill{\caps\SAuthor}\hfill\fi}%
\makeatother

\def\TSkip{\bigskip}
\newbox\TheTitle{\obeylines\gdef\GetTitle #1
\ShortTitle  #2
\SubTitle    #3
\Author      #4
\ShortAuthor #5
\EndTitle
{\setbox\TheTitle=\vbox{\baselineskip=20pt\let\par=\cr\obeylines%
\halign{\centerline{\Caps##}\cr\noalign{\medskip}\cr#1\cr}}%
	\copy\TheTitle\TSkip\TSkip%
\def\next{#2}\ifx\next\empty\gdef\STitle{#1}\else\gdef\STitle{#2}\fi%
\def\next{#3}\ifx\next\empty%
    \else\setbox\TheTitle=\vbox{\baselineskip=20pt\let\par=\cr\obeylines%
    \halign{\centerline{\caps##} #3\cr}}\copy\TheTitle\TSkip\TSkip\fi%
\centerline{\caps #4}\TSkip\TSkip%
\def\next{#5}\ifx\next\empty\gdef\SAuthor{#4}\else\gdef\SAuthor{#5}\fi%
\ifx\received\empty\relax
    \else\centerline{\eightrm Received: \received}\fi%
\ifx\revised\empty\TSkip%
    \else\centerline{\eightrm Revised: \revised}\TSkip\fi%
\ifx\communicated\empty\relax
    \else\centerline{\eightrm Communicated by \communicated}\fi\TSkip\TSkip%
\catcode'015=5}}\def\Title{\obeylines\GetTitle}
\def\Abstract{\begingroup\narrower
    \parskip=\medskipamount\parindent=0pt{\caps Abstract. }}
\def\EndAbstract{\par\endgroup\TSkip}

\long\def\MSC#1\EndMSC{\def\arg{#1}\ifx\arg\empty\relax\else
     {\par\narrower\noindent%
     2020 Mathematics Subject Classification: #1\par}\fi}

\long\def\KEY#1\EndKEY{\def\arg{#1}\ifx\arg\empty\relax\else
	{\par\narrower\noindent Keywords and Phrases: #1\par}\fi\TSkip}

\newbox\TheAdd\def\Addresses{\vfill\copy\TheAdd\vfill
    \ifodd\number\lastpage\vfill\eject\phantom{.}\vfill\eject\fi}
{\obeylines\gdef\GetAddress #1
\Address #2 
\Address #3
\Address #4
\EndAddress
{\def\xs{4.3truecm}\parindent=0pt
\setbox0=\vtop{{\obeylines\hsize=\xs#1\par}}\def\next{#2}
\ifx\next\empty 
     \setbox\TheAdd=\hbox to\hsize{\hfill\copy0\hfill}
\else\setbox1=\vtop{{\obeylines\hsize=\xs#2\par}}\def\next{#3}
\ifx\next\empty 
     \setbox\TheAdd=\hbox to\hsize{\hfill\copy0\hfill\copy1\hfill}
\else\setbox2=\vtop{{\obeylines\hsize=\xs#3\par}}\def\next{#4}
\ifx\next\empty\ 
     \setbox\TheAdd=\vtop{\hbox to\hsize{\hfill\copy0\hfill\copy1\hfill}
                \vskip20pt\hbox to\hsize{\hfill\copy2\hfill}}
\else\setbox3=\vtop{{\obeylines\hsize=\xs#4\par}}
     \setbox\TheAdd=\vtop{\hbox to\hsize{\hfill\copy0\hfill\copy1\hfill}
	        \vskip20pt\hbox to\hsize{\hfill\copy2\hfill\copy3\hfill}}
\fi\fi\fi\catcode'015=5}}\gdef\Address{\obeylines\GetAddress}

\def\LOCAL{\jobname.files}
\newtheorem{definition}{Definition}[section]
\newtheorem{theorem}[definition]{Theorem}

\newtheorem{coro}[definition]{Corollary}
\newtheorem{proposition}[definition]{Proposition}
\newtheorem{remark}[definition]{Remark}
\begin{document}
\Title Projective Bundle Theorem
in MW-Motivic Cohomology
\ShortTitle
\SubTitle   
\Author Nanjun Yang
\ShortAuthor 
\EndTitle
\Abstract We present a version of projective bundle theorem in MW-motives (resp. Chow-Witt rings), which says that \(\widetilde{CH}^*(\mathbb{P}(E))\) is determined by \(\widetilde{CH}^*(X)\), \(\widetilde{CH}^*(X,det(E)^{\vee})\), \(CH^*(X)\) and \(Sq^2\) for smooth quasi-projective schemes \(X\) and vector bundles \(E\) over \(X\) with \(e(E^{\vee})=0\in H^n(X,W(det(E)))\), provided that \(_2CH^*(X)=0\).

As an application, we compute the MW-motives of blow-ups with smooth centers. Moreover, we discuss the invariance of Chow-Witt cycles of projective bundles under automorphisms of vector bundles.
\EndAbstract
\MSC Primary: 11E81, 14F42
\EndMSC
\KEY MW-motivic cohomology, Chow-Witt ring, Projective bundle theorem
\EndKEY
\Address Nanjun Yang\\Yau Math. Sci. Center\\Jing Zhai\\Tsinghua University\\Hai Dian District\\Beijing China\\ynj.t.g@126.com
\Address
\Address
\Address
\EndAddress
%
%
%
%
%
%
%

\section{Introduction}
MW-motivic cohomology is the Chow-Witt counterpart of the ordinary motivic cohomology defined by V. Voevodsky, which was developed by B. Calm\`es, F. D\'eglise and J. Fasel. It is well-known that the usual projective bundle theorem (see \cite[Theorem 15.12]{MVW}) does not hold for MW-motives (see \cite[Remark 5.6]{Y}) so the computation of \(\widetilde{CH}^*(\mathbb{P}(E))\) for vector bundles \(E\) becomes nontrivial. In \cite[Theorem 9.1, 9.2, 9.4 and Corollary 11.8]{F}, J. Fasel computed the cohomology groups \(H^i(\mathbb{P}(E),K_j^W,\mathscr{L})\) and \(H^i(\mathbb{P}^n,{K}_j^{MW},\mathscr{L})\), leaving the question of Chow-Witt groups of projective bundles open.

This paper gives a possible approach to this question, by splitting the MW-motives of projective bundles. Suppose \(S\in Sm/k\), \(pt=Spec(k)\) and \(R\) is a commutative ring. Denote by \(\widetilde{DM}^{eff}(S,R)\) (resp. \(\widetilde{DM}(S,R)\)) the category of effective (resp. stable) MW-motives over \(S\) with coefficients in \(R\) (see Section \ref{Four Motivic Theories}) and by \(R(X)\) the motive of \(X\in Sm/S\) with coefficient in \(R\). The first step is to split that of projective spaces, which is the following (see Theorem \ref{splitting1}):
\begin{theorem}
Suppose \(n\in\mathbb{N}\) and \(p:\mathbb{P}^n\longrightarrow pt\) is the structure map.
\begin{enumerate}
\item If \(n\) is odd, the morphism
\[\Theta:R(\mathbb{P}^n)\longrightarrow R\oplus\bigoplus_{i=1}^{\frac{n-1}{2}}R/{\eta}(2i-1)[4i-2]\oplus R(n)[2n]\]
is an isomorphism in \(\widetilde{DM}^{eff}(pt,R)\) where
\[\Theta=(p,c^{2i-1}_n\otimes R,{\tau}_{n+1}\otimes R).\]
\item If \(n\) is even, the morphism
\[\Theta:R(\mathbb{P}^n)\longrightarrow R\oplus\bigoplus_{i=1}^{\frac{n}{2}}R/{\eta}(2i-1)[4i-2]\]
is an isomorphism in \(\widetilde{DM}^{eff}(pt,R)\) where
\[\Theta=(p,c^{2i-1}_n\otimes R).\]
\end{enumerate}
\end{theorem}
Here \({\tau}_{n+1}=i_*(1)\) for some rational point \(i:pt\longrightarrow\mathbb{P}^n\) and \(\eta:\mathbb{G}_m\longrightarrow pt\) is the Hopf map (see Definition \ref{cone}), where we set \(A/\eta=A\otimes C(\eta)\) for every motive \(A\). The morphism \(c^{2i-1}_n\) is the unique one which corresponds to \((c_1(O_{\mathbb{P}^n}(1))^{2i-1}, c_1(O_{\mathbb{P}^n}(1))^{2i})\) as in usual case (see Proposition \ref{gen}, (2) and \cite[2.10]{D}). This recovers \cite[Corollary 11.8]{F}:
\begin{coro}
Suppose \(0\leq i\leq n\), we have
\[\widetilde{CH}^{i}(\mathbb{P}^n)=\begin{cases}GW(k)&\text{if \(i=0\) or \(i=n\) and \(n\) is odd}\\\mathbb{Z}&\text{if \(i>0\) is even}\\2\mathbb{Z}&\text{else}.\end{cases}\]
\end{coro}

To obtain the projective bundle theorem, the second step is to identify the group
\[\eta_{MW}^i(X):=Hom_{\widetilde{DM}^{eff}(pt)}(\mathbb{Z}(X),\mathbb{Z}/{\eta}(i)[2i])\]
by the structure theorem (see \cite[Theorem 17]{B}) of the spectrum \(H\widetilde{\mathbb{Z}}\) representing the MW-motivic cohomology in the motivic stable homotopy category \(\mathcal{SH}(pt)\) (see Section \ref{Motivic Stable Homotopy Category and Slice Filtrations}). Our result is the following (see Theorem \ref{eta} and Corollary \ref{etaq}):
\begin{theorem}
If \(_2CH^{i+1}(X)=0\), we have a natural Cartesian square
\[
	\xymatrix
	{
		\eta^i_{MW}(X)\ar[r]\ar[d]		&CH^{i+1}(X)\ar[d]\\
		CH^i(X)\ar[r]^-{Sq^2\circ\pi}	&CH^{i+1}(X)/2
	}
\]
where \(\pi\) is the modulo \(2\) map. Under \(\mathbb{Z}[\frac{1}{2}]\)-coefficients, we have a natural isomorphism
\[CH^i(X)[\frac{1}{2}]\oplus CH^{i+1}(X)[\frac{1}{2}]\cong\eta_{MW}^i(X,\mathbb{Z}[\frac{1}{2}])\]
for every \(X\in Sm/k\).
\end{theorem}
This allows us to define a global version \(c^{2i-1}_{rk(E)}(E)\) of \(c^{2i-1}_{rk(E)}\) for vector bundles \(E\) (see Proposition \ref{global}).

Our main theorem is the following (see Corollary \ref{pbtint}):
\begin{theorem}
Let \(S\in Sm/k\), \(X\in Sm/S\) be quasi-projective and \(E\) be a vector bundle of rank \(n\) over \(X\).
\begin{enumerate}
\item We have isomorphisms
\[R(\mathbb{P}(E))/\eta\cong\bigoplus_{i=0}^{n-1}R(X)/\eta(i)[2i]\]
\[Th(E)/\eta\cong R(X)/\eta(n)[2n]\]
in \(\widetilde{DM}(S,R)\).
\item If \(n\) is odd, we have an isomorphism
\[R(\mathbb{P}(E))\cong R(X)\oplus\bigoplus_{i=1}^{\frac{n-1}{2}}R(X)/{\eta}(2i-1)[4i-2]\]
in \(\widetilde{DM}^{eff}(S,R)\).

In particular, we have (\(k=min\{\lfloor\frac{i+1}{2}\rfloor,\frac{n-1}{2}\}\))
\[\widetilde{CH}^{i}(\mathbb{P}(E))=\widetilde{CH}^i(X)\oplus\bigoplus_{j=1}^{k}\eta_{MW}^{i-2j}(X).\]
\item If \(n\) is even, we have a distinguished triangle
\begin{footnotesize}
\[Th(det(E)^{\vee})(n-2)[2n-4]\to R(\mathbb{P}(E))\longrightarrow R(X)\oplus\bigoplus_{i=1}^{\frac{n}{2}-1}R(X)/{\eta}(2i-1)[4i-2]\xrightarrow{\mu}\cdots[1]\]
\end{footnotesize}
in \(\widetilde{DM}(S,R)\) where
\(\mu=0\) if \(e(E^{\vee})=0\in H^n(X,W(det(E)))\). In this case, we have
\[\widetilde{CH}^{i}(\mathbb{P}(E))=\widetilde{CH}^i(X)\oplus\bigoplus_{j=1}^{\frac{n}{2}-1}\eta_{MW}^{i-2j}(X)\oplus\widetilde{CH}^{i-n+1}(X,det(E)^{\vee}).\]
\end{enumerate}
\end{theorem}
Note that for any vector bundle \(E\) of odd rank \(n\), we have \(e(E^{\vee})=0\in H^n(X,W(det(E)))\) by \cite[Proposition 3.4]{L}. So the theorem above just says that there is a projective bundle theorem for any \(E\) if its Euler class vanishes in Witt group.

This in turn yields a computation of the MW-motives (resp. Chow-Witt ring) of blow-ups with smooth centers, as the following (see Theorem \ref{blowup}):
\begin{theorem}
Suppose \(S\in Sm/k\) and \(X, Z\in Sm/S\) with \(Z\) being closed in \(X\). If \(n:=codim_X(Z)\) is odd and \(Z\) is quasi-projective, we have
\[R(Bl_Z(X))\cong R(X)\oplus\bigoplus_{i=1}^{\frac{n-1}{2}}R(Z)/{\eta}(2i-1)[4i-2]\]
in \(\widetilde{DM}^{eff}(S,R)\).

In particular, we have (\(k=min\{\lfloor\frac{i+1}{2}\rfloor,\frac{n-1}{2}\}\))
\[\widetilde{CH}^{i}(Bl_Z(X))=\widetilde{CH}^i(X)\oplus\bigoplus_{j=1}^{k}\eta_{MW}^{i-2j}(Z).\]
\end{theorem}
Last but not the least, given an automorphism of \(E\), we study its action on the MW-motive of \(\mathbb{P}(E)\) (see Theorem \ref{oriinv2}).
\begin{theorem}
Let \(S\in Sm/k\), \(X\in Sm/S\) be quasi-projective with \(_2CH^*(X)=0\), \(E\) be a vector bundle of rank \(n\) over \(X\) and \(\varphi\in Aut_{O_X}(E)\).
\begin{enumerate}
\item If \(n\) is odd, the morphism \(\mathbb{P}(\varphi)=Id\) in \(\widetilde{DM}^{eff}(S,R)\).
\item If \(n\geq dim(X)+2\) is even and \(e(E^{\vee})=0\in H^n(X,W(det(E)))\), the morphism \(\mathbb{P}(\varphi)=Id\) in \(\widetilde{DM}(S,R)\) if \(\varphi\) is \(SL^c\) (see Definition \ref{slc}).
\end{enumerate}
\end{theorem}
\textbf{Organization of this article. }We explain in Section \ref{Four Motivic Theories} the basic definitions and properties of the motivic categories (denoted by \(DM_{{K}}\)) corresponding to \({K}_*={K}^{MW}_*\), \({K}^M_*\), \({K}^W_*\) and \({K}^M_*/2\). Section \ref{Motivic Stable Homotopy Category and Slice Filtrations} is devoted to a brief introduction to the stable \(\mathbb{A}^1\)-homotopy category \(\mathcal{SH}(pt)\) and the slice filtrations of spectra.

We discuss independently in Section \ref{Mapping Cone of the Hopf Map} the computation of \(\eta_{MW}^n(X)\). The point is that even if \(H\widetilde{\mathbb{Z}}/\eta\) could not split into components coming from Chow motives (see Remark \ref{non-split}), \(\eta_{MW}^n(X)\) parameterizes the pair of cycles where one is the Steenrod square of the other for some \(X\), which allows us to find out `Chern classes' and thus to adopt the pattern of the original projective bundle formula to the Chow-Witt case.

We prove the projective bundle theorem in Section \ref{The Projective Bundle Theorem}. The first step is to compute the MW-motive of projective spaces. Then global formula for odd rank bundles follows from the `Chern classes' obtained from Section \ref{Mapping Cone of the Hopf Map}. While for even rank bundles, we use its projective completion and the result of odd rank bundles to obtain a distinguished triangle, which splits if its Euler class vanishes in Witt group. Furthermore, for arbitrary vector bundle, we show that there is an easy projective bundle formula after tensoring with \(\mathbb{Z}/\eta\). 

As a consequence, we compute the MW-motives of blow-ups with smooth centers if the codimension is odd. We also discover when automorphisms of vector bundles acts trivially on MW-motives of projective bundles.

\textbf{Acknowledgements. }The author would like to thank Alexey Ananyevskiy, Tom Bachmann, Jean Fasel, Yong Yang and Dingxin Zhang for helpful discussions. The careful suggestions of the referees are also greatly appreciated.

\textbf{Conventions. }We denote by \(pt=Spec(k)\) where \(k\) is an infinite perfect field with \(char(k)\neq 2\). Define \(Sm/S\) to be the category of separated and smooth schemes over \(S\) (see \cite[page 268]{Ha}). We say \(X\in Sm/S\) is quasi-projective if it is quasi-projective over \(k\). We set \(C(f)\) to be the mapping cone of \(f\). Define
\[\begin{array}{cc}Hom_{SH(pt)}(-,-)=[-,-],&Hom_{DM_{K}^{eff}(pt)}(-,-)=[-,-]_{{K}}.\end{array}\]
\section{Four Motivic Theories}\label{Four Motivic Theories}
The main references of this section are \cite{BCDFO} and \cite{Y1}. In the sequel, we denote by \({K}_*\) one of the homotopy modules \({K}^{MW}_*, {K}^M_*, {K}^W_*, {K}^M_*/2\). They are related by the Cartesian square
\[
	\xymatrix
	{
		{K}_*^{MW}\ar[r]\ar[d] 	&{K}^W_*\ar[d]\\
		{K}_*^M\ar[r]				&{K}^M_*/2
	}.
\]
In \cite{BCDFO}, they mainly developped the motivic theory when \(K_*=K_*^{MW}\), which was inherited from Voevodsky's construction, namely the case when \(K_*=K_*^M\). But the method could be generalized to \(K_*=K_*^W, K_*^{M}/2\) without significant difficulty. We refer the readers to \cite{Y1} for axiomatic approach of correspondence theory and its associated motivic theory. 

Let {\(m\in\mathbb{Z}\)}, \(F/k\) be a finitely generated field extension of the base field \(k\) and \(L\) be a one-dimensional $F$-vector space. One can define \({K}_m(F,L)\) as in \cite[Remark 2.21]{Mo}. If \(X\) is a smooth scheme, \(\mathscr{L}\) is a line bundle over \(X\) and \(y\in X\), we set 
\[
\widetilde{{K}}_m(k(y),\mathscr{L}):={K}_m(k(y),\Lambda_y^*\otimes_{k(y)}\mathscr{L}_y),
\]
where \(k(y)\) is the residue field of \(y\) and \(\Lambda_y^*\) is the top exterior power of the tangent space of \(y\). If \({K_*}=K^M_*, K^M_*/2\), \(\widetilde{{K}}_m(k(y),\mathscr{L})\) is independent of \(\mathscr{L}\) up to a canonical isomorphism. If \(T\subset X\) is a closed set and \(n\in\mathbb{N}\), define
\[C_{RS,T}^n(X;{K}_m;\mathscr{L})=\bigoplus_{y\in X^{(n)}\cap T}\widetilde{{K}}_{m-n}(k(y),\mathscr{L}),\]
where \(X^{(n)}\) means the points of codimension \(n\) in \(X\). Then \(C_{RS,T}^*(X;{K}_m;\mathscr{L})\) forms a complex (see \cite[Definition 4.11]{Mo}, \cite[Remark 4.13]{Mo}, \cite[Theorem 4.31]{Mo} and \cite[D\'efinition 10.2.11]{F1}), which is called the Rost-Schmid complex with support on \(T\). Define (see \cite[Definition 4.1.1, \S 2]{BCDFO} and \cite[Definition 5.2]{Y1})
\[{K}CH^n_T(X,\mathscr{L})=H^n(C^*_{RS,T}(X;{K}_n;\mathscr{L})).\]
Thus we see that \({K}CH=\widetilde{CH}, CH, Ch\) when \({K_*}=K^{MW}_*, K^M_*, K^M_*/2\) respectively.

Suppose \(S\in Sm/k\). For any \(X, Y\in Sm/S\), define \(\mathscr{A}_S(X,Y)\) to be the poset of closed subsets in \(X\times_SY\) such that each of its irreducible component is finite over \(X\) and of dimension \(dimX\). Suppose \(R\) is a commutative ring. Let
\[{K}Cor_S(X,Y,R):=\varinjlim_T{K}CH_T^{dimY-dimS}(X\times_SY,\omega_{X\times_SY/X})\otimes_{\mathbb{Z}}R,\]
be the finite correspondences between \(X\) and \(Y\) over \(S\) with coefficients in \(R\), where \(T\in\mathscr{A}_S(X,Y)\). Hence \({K}Cor\) just means \(\widetilde{Cor}\), \(Cor\) and \(WCor\)  (see \cite[\S 3]{BCDFO}) when \({K_*}=K^{MW}_*, K^M_*, K^W_*\) respectively. This produces an additive category \({K}Cor_S\) whose objects are the same as \(Sm/S\) and whose morphisms are defined above. There is a functor \(\alpha:Sm/S\longrightarrow {K}Cor_S\) sending a morphism to its graph (see \cite[1.1.6, \S 2]{BCDFO}, \cite[Definition 5.3]{Y1}).

We define a presheaf with \({K}\)-transfers to be a contravariant additive functor from \({K}Cor_S\) to \(Ab\). It is a sheaf with \({K}\)-transfers if it is a Nisnevich sheaf after restricting to \(Sm/S\) via \(\alpha\). For any smooth scheme $X$, let $c(X)$ be the representable presheaf with \({K}\)-transfers of \(X\).

Let  \(PSh_{{K}}(S,R)\) be the category of presheaves with \({K}\)-transfers over \(S\) and let \(Sh_{{K}}(S,R)\) be the full subcategory of sheaves with \({K}\)-transfers (see \cite[Definition 1.2.1 and Definition 1.2.4, \S 3]{BCDFO}). Both categories are abelian and have enough injectives (\cite[Proposition 1.2.11, \S 3]{BCDFO}). There is a sheafification functor (see \cite[Proposition 1.2.11, \S 3]{BCDFO})
\[a:PSh_{{K}}(S,R)\longrightarrow Sh_{{K}}(S,R)\]
and we set \(R(X)=a(c(X))\), which is called the motive of \(X\) with \(R\)-coefficients. The Tate twist is denoted by \(R(1)\). For any vector bundle \(E\) on \(X\in Sm/S\), we define \(Th(E)=R(E)/R(E^{\times})\) to be the Thom space of \(E\).

Denote by \(C(Sh_{{K}}(S,R))\) (resp. \(D_{{K}}(S,R)\)) the (resp. derived) category of cochain complexes of sheaves with \({K}\)-transfers. Define the category of \({K}\)-motives over \(S\) with coefficients in \(R\)
\[DM_{K}^{eff}(S,R)=D_{{K}}(S,R)[(R(X\times\mathbb{A}^1)\longrightarrow R(X))^{-1}]\]
for any \(X\in Sm/S\).

Denote by \(Sp(S,R)\) the category of (symmetric) \(\mathbb{G}_m\)-spectra, whose homotopy category is denoted by \(DM_{{K}}(S,R)\). There is an adjunction
\[\Sigma^{\infty}:DM_{{K}}^{eff}(S,R)\rightleftharpoons DM_{{K}}(S,R):\Omega^{\infty}\]
by \cite[Example 7.15]{CD}. We will ignore \(\Sigma^{\infty}\) if there is no ambiguity of underlying category.

We will write \(\widetilde{DM}\) (resp. \(DM\)) for \(DM_{MW}\) (resp. \(DM_M\)) and \(DM_{K}^{eff}(S)\) (resp. \(DM_{{K}}(S)\)) for \(DM_{K}^{eff}(S,\mathbb{Z})\) (resp. \(DM_{{K}}(S,\mathbb{Z})\)) accordingly.
\begin{proposition}\label{derived}
\begin{enumerate}
\item The category \(DM_{K}^{eff}(S,R)\) is symmetric monoidal with \(R(X)\otimes_SR(Y)=R(X\times_SY)\) for any \(X, Y\in Sm/S\).
\item Suppose \(f:S\longrightarrow T\) is a morphism in \(Sm/k\). There is an adjunction
\[f^*:DM_{K}^{eff}(T,R)\rightleftharpoons DM_{K}^{eff}(S,R):f_*\]
satisfying \(f^*R(Y)=R(Y\times_TS)\) for any \(Y\in Sm/T\).
\item Suppose \(f:S\longrightarrow T\) is a smooth morphism in \(Sm/k\). There is an adjunction
\[f_{\#}:DM_{K}^{eff}(S,R)\rightleftharpoons DM_{K}^{eff}(T,R):f^*\]
satisfying \(f_{\#}R(X)=R(X)\) for any \(X\in Sm/S\). Moreover, for any \(F\in DM_{K}^{eff}(T,R)\) and \(G\in DM_{K}^{eff}(S,R)\), we have
\[f_{\#}(G\otimes_Sf^*F)=(f_{\#}G)\otimes_TF.\]
\end{enumerate}
The same properties hold for \(DM_{{K}}(S,R)\).
\end{proposition}
\begin{proof}
See \cite[Proposition 5.25, Proposition 6.44 and Proposition 6.45]{Y1} and \cite{CD}.
\end{proof}
The readers refer to \cite[Corollary 2.21]{Y} for compability between Thom spaces and operations (\(f^*\), \(f_{\#}\), \(\otimes\)) above.
\begin{proposition}
Suppose \(\varphi:R\longrightarrow R'\) is a ring morphism and \(X\in Sm/S\).
\begin{enumerate}
\item We have adjunctions
\[\varphi^*:DM_{K}^{eff}(S,R)\rightleftharpoons DM_{K}^{eff}(S,R'):\varphi_*,\]\[\varphi^*:DM_{{K}}(S,R)\rightleftharpoons DM_{{K}}(S,R'):\varphi_*\]
where \(\varphi^*\) is monoidal and \(\varphi^*R(X)=R'(X)\).
\item If \(\varphi\) is flat, we have
\[Hom_{DM_{K}^{eff}(S,R')}(R'(X),\varphi^*C)=Hom_{DM_{K}^{eff}(S,R)}(R(X),C)\otimes_RR',\]
\[Hom_{DM_{{K}}(S,R')}(\Sigma^{\infty}R'(X),\varphi^*E)=Hom_{DM_{{K}}(S,R)}(\Sigma^{\infty}R(X),E)\otimes_RR'\]
for \(C\in DM_{K}^{eff}(S,R)\) and \(E\in DM_{{K}}(S,R)\).
\end{enumerate}
\end{proposition}
\begin{proof}
\begin{enumerate}
\item We have by definition a morphism
\[\varphi:{K}Cor_S(X,Y,R)\longrightarrow {K}Cor_S(X,Y,R')\]
for every \(X, Y\in Sm/S\). For every \(X\in Sm/S\) and \(F\in Sh_{{K}}(S,R)\), \(F(X)\) is naturally an \(R\)-module. So we define \(a^*F\) to be the sheafification of the presheaf
\[X\longmapsto F(X)\otimes_RR',\]
which induces an adjunction 
\[a^*:Sh_{{K}}(S,R)\rightleftharpoons Sh_{{K}}(S,R'):a_*\]
where \(a_*\) is the restriction functor. The functor \(a^*\) is monoidal and satisfies
\[a^*R(X)=R'(X)\]
for any \(X\in Sm/S\). So it satisfies the conditions in \cite[2.4 and Lemma 4.8]{CD} by a similar argument as in \cite[Proposition 6.2]{Y} and gives us the adjunction between \(DM_{K}^{eff}\). The stable case follows from \cite[Proposition 5.5]{H}.
\item The functor \(\varphi^*\) is exact and takes flabby sheaves to flabby sheaves (see \cite[Proposition 2.12]{M}) hence
\[H^i(X,\varphi^*F)=H^i(X,F)\otimes_RR'\]
for every \(F\in Sh_{{K}}(S,R)\), \(i\in\mathbb{N}\) and \(X\in Sm/S\). This implies
\[Hom_{D_{{K}}(S,R')}(R'(X),\varphi^*C)=Hom_{D_{{K}}(S,R)}(R(X),C)\otimes_RR'\]
for every \(C\in D_{{K}}^+(S,R)\) by the hypercohomology spectral sequence. Then the unbounded cases follow from \(C=\varinjlim_nC^{\geq n}\) and the Five Lemma. Hence the functor \(\varphi^*\) preserves fibrant objects, \(\mathbb{A}^1\)-weak equivalences (\(\varphi^*\) is left Quillen and preserves homotopy condition) and \(\mathbb{A}^1\)-local objects. Then we obtain the first equation. The second equation follows from the first equation since \(\varphi^*\) preserves stable \(\mathbb{A}^1\)-weak equivalences and \(\Omega^{\infty}\)-spectra (see \cite[7.8]{CD}). 
\end{enumerate}
\end{proof}
\begin{proposition}\label{dmadj}
In the diagram
\[
	\xymatrix
	{
		{K}_*^{MW}\ar[r]\ar[d]\ar[rd] 	&{K}_*^W\ar[d]\\
		{K}_*^M\ar[r] 							&{K}_*^M/2
	},
\]
each arrow \(f:K_1\longrightarrow K_2\) induces adjoint pairs
\[f^*:DM_{K_1}^{eff}(S,R)\rightleftharpoons DM_{K_2}^{eff}(S,R):f_*\]
\[f^*:DM_{K_1}(S,R)\rightleftharpoons DM_{K_2}(S,R):f_*\]
between corresponding motivic categories. Moreover, \(f^*\) is monoidal and compatible with change of coefficients.
\end{proposition}
\begin{proof}
The map \(f\) induces a map \(f:K_1Cor_S(X,Y,R)\longrightarrow K_2Cor_S(X,Y,R)\) hence a functor \(f:K_1Cor_S\longrightarrow K_2Cor_S\). Then we obtain an adjunction
\[f^*:Sh_{K_1}(S,R)\rightleftharpoons Sh_{K_2}(S,R):f_*\]
by \cite[Lemma 5.7]{Y1}, where \(f_*\) is the restricion via \(f\). They are also Quillen functors of the model structure of unbounded complexes because \(f^*\) preserves representable sheaves. Hence we obtain the adjoint pairs desired.
The method is essentially the same as \cite[3.2.4 and 3.3.6.a \S 3]{BCDFO}.
\end{proof}
\begin{proposition}\label{zarsep}
Let \(X\in Sm/k\), \(\{U_i\}\) be an open covering of \(X\) and \(f\) be a morphism in {\(DM^{eff}_{{K}}(X,R)\)} or \(DM_{{K}}(X,R)\). If \(f|_{U_i}\) is an isomorphism for every \(i\), \(f\) is an isomorphism.
\end{proposition}
\begin{proof}
For the effective case, the proof is the same as in \cite[Proposition 2.16]{Y}. For the stable case, we use that \(\{R(X)(n)\}\) for \(X\in Sm/k\) and \(n\in\mathbb{Z}\) is a system of generators in \(DM_K(S,R)\) (see \cite[11.1.6]{CD1}).
\end{proof}
\begin{proposition}\label{cancellation}
Suppose \({K_*}=K^{MW}_*, K^M_*\). Let \(U, V\in DM_{K}^{eff}(pt)\). The map
\[\xymatrix{[U,V]_{{K}}\ar[r]^-{\otimes\mathbb{Z}(i)}&[U(i),V(i)]_{{K}}}, i>0\]
is an isomorphism.
\end{proposition}
\begin{proof}
See \cite[Corollary 4.10]{V} and \cite[Theorem 4.0.1 \S 4]{BCDFO}.
\end{proof}
As a consequence, the infinite suspension functor \(\Sigma^{\infty}\) induces fully faithful embeddings \(DM^{eff}(pt)\subseteq DM(pt)\) and \(\widetilde{DM}^{eff}(pt)\subseteq\widetilde{DM}(pt)\) (see \cite[Proposition 5.3.25]{CD1}).

Suppose \(X\in Sm/k\) and \(E\) is a vector bundle of rank \(n\) over \(X\). Recall that the Euler class (see \cite[Chapitre 13]{F1}) \(e(E)\in\widetilde{CH}^n(X,det(E)^{\vee})\) is the image of \(1\) under the composite
\[\widetilde{CH}^0(X)\xrightarrow{z_*}\widetilde{CH}^n(E,p^*det(E)^{\vee})\xrightarrow{(p^*)^{-1}}\widetilde{CH}^n(X,det(E)^{\vee})\]
where \(p:E\longrightarrow X\) is the structure map and \(z\) is the zero section. Its image in \(H^n(X,W(det(E)^{\vee}))\) is also denoted by \(e(E)\).

A result we will frequently use is that for any \(X\in Sm/k\) and \(E\) being vector bundle on \(X\),
\[Th(E)\otimes Th(E^{\vee})=\mathbb{Z}(X)(2rkE)[4rkE]\]
in \(\widetilde{DM}(X)\) by \cite[Theorem 6.1]{Y1}. In particular, \(Th(E)\) is invertible.
\begin{proposition}\label{gysin}
Let \(S\in Sm/k\), \(X\in Sm/S\) and \(E\) be a vector bundle of rank \(n\) over \(X\). We have
\[Th(E)=Th(det(E))(n-1)[2n-2]\]
in \(\widetilde{DM}(S)\) and the composite
\[\pi:\mathbb{Z}(X)\xrightarrow{z}\mathbb{Z}(E)\longrightarrow Th(E)\cong Th(det(E))(n-1)[2n-2]\]
in \(\widetilde{DM}(X)\) is given by the Euler class \(e(E)\in\widetilde{CH}^n(X,det(E^{\vee}))\) where \(z\) is the zero section.
\end{proposition}
\begin{proof}
We have the isomorphism
\[Th:Th(E\oplus det(E)^{\vee})\longrightarrow\mathbb{Z}(n+1)[2n+2]\]
in \(\widetilde{DM}(X)\) given by the Thom class of \(E\oplus det(E)^{\vee}\) by \cite[Theorem 6.1]{Y}. By the same proof as in \cite[Theorem 6.2]{Y}, we obtain the first statement.

For the second statement, there is a commutative diagram
\[\footnotesize
	\xymatrix
	{
		[Th(E\oplus det(E)^{\vee}),\mathbb{Z}(n+1)[2n+2]]\ar[r]^{\otimes Th(det(E))(-2)[-4]}\ar[d]_{\alpha}	&[Th(E),Th(det(E))(n-1)[2n-2]]\ar[d]_{\beta}\\
		[Th(det(E)^{\vee}),\mathbb{Z}(n+1)[2n+2]]\ar[r]^{\otimes Th(det(E))(-2)[-4]}						&[\mathbb{Z}(X),Th(det(E))(n-1)[2n-2]]
	}
\]
where \([-,-]\) means the Hom-group in \(\widetilde{DM}(X)\). Suppose \(z_1\) (resp. \(z_2\)) is the zero section of \(E\) (resp. \(det(E)^{\vee}\)). The upper left term in the diagram is identified with \(\widetilde{CH}^{n+1}_{z_1\times_Xz_2}(E\oplus det(E)^{\vee})\) and we have the identification
\[[Th(det(E)^{\vee}),\mathbb{Z}(n+1)[2n+2]]=\widetilde{CH}^{n+1}_{z_2}(det(E)^{\vee})=\widetilde{CH}^n(X,det(E)^{\vee}).\]

The \(\alpha\) and \(\beta\) come from the composites
\[Th(det(E)^{\vee})\xrightarrow{z}\mathbb{Z}(E)\otimes Th(det(E)^{\vee})\longrightarrow Th(E\oplus det(E)^{\vee})\]
and
\[\mathbb{Z}(X)\xrightarrow{z}\mathbb{Z}(E)\longrightarrow Th(E)\]
respectively. In terms of Chow-Witt groups, \(\alpha\) is the composite
\[\widetilde{CH}^{n+1}_{z_1\times_Xz_2}(E\oplus det(E)^{\vee})\longrightarrow\widetilde{CH}^{n+1}_{E\times_Xz_2}(E\oplus det(E)^{\vee})\xrightarrow[\cong]{(p^*)^{-1}}\widetilde{CH}^{n+1}_{z_2}(det(E)^{\vee})\]
where \(p:E\oplus det(E)^{\vee}\longrightarrow det(E)^{\vee}\) is the projection and \(p^*\) is an isomorphism by homotopy invariance. Finally, we have \(\beta(Th\otimes Th(det(E))(-2)[-4])=\pi\).

Now we look at the commutative diagram
\[
	\xymatrix
	{
																																&\widetilde{CH}^0(X)\ar[ld]\ar[d]\\
		\widetilde{CH}^{n+1}_{z_1\times_Xz_2}(E\oplus det(E)^{\vee})\ar[d]						&\widetilde{CH}^n_{z_2}(E,q^*det(E)^{\vee})\ar[l]_-{\cong}\ar[d]\\
		\widetilde{CH}^{n+1}_{E\times_Xz_2}(E\oplus det(E)^{\vee})\ar[d]_{(p^*)^{-1}}	&\widetilde{CH}^n(E,q^*det(E)^{\vee})\ar[l]_-{\cong}\ar[d]_{(q^*)^{-1}}\\
		\widetilde{CH}^{n+1}_{z_2}(det(E)^{\vee})															&\widetilde{CH}^n(X,det(E)^{\vee})\ar[l]_{\cong}
	}
\]
where \(q:E\longrightarrow X\) is the structure map and horizontal arrows are given by push forwards. So we see that \(\pi\in\widetilde{CH}^n(X,det(E)^{\vee})\) just comes from the image of \(1\in\widetilde{CH}^0(X)\) along the right vertical arrows. So it is equal to \(e(E)\) by definition.
\end{proof}
\section{Motivic Stable Homotopy Category and Slice Filtrations}\label{Motivic Stable Homotopy Category and Slice Filtrations}
The main references of this section are \cite{Mo1}, \cite{B} and \cite{DLORV}. Let \(sShv_{\bullet}(Sm/k)\) be the category of pointed simplicial sheaves over \(Sm/k\) for the Nisnevich topology. We localize it by the morphisms
\[F\wedge\mathbb{A}^1_+\longrightarrow F\]
for every \(F\in sShv_{\bullet}(Sm/k)\), obtaining the homotopy category \(\mathcal{H}_{\bullet}(pt)\).

Define \(\mathcal{SH}(pt)\) to be the homotopy category of the \(\mathbb{P}^1\) or \(\mathbb{G}_m-S^1\) spectra of \(sShv_{\bullet}(Sm/k)\) (see \cite[2.3]{DLORV}). It is triangulated and symmetric monoidal with respect to the smash product \(\wedge\). There is an adjunction
\[\Sigma^{\infty}:\mathcal{H}_{\bullet}(pt)\rightleftharpoons\mathcal{SH}(pt):\Omega^{\infty}.\]

For any \(E\in\mathcal{SH}(pt)\) and \(n,m\in\mathbb{Z}\), define \(\pi_n(E)_m\) to be the (Nisnevich) sheafification of the presheaf on \(Sm/k\)
\[X\longmapsto[\Sigma^{\infty}X_+\wedge S^n,E(m)[m]]_{\mathcal{SH}(pt)}.\]
Then a morphism \(E_1\longrightarrow E_2\) between spectra is a weak equivalence if and only if it induces an isomorphism
\[\pi_n(E_1)_m\cong\pi_n(E_2)_m\]
for every \(n,m\in\mathbb{Z}\).

Define
\[\mathcal{SH}_{\leq -1}=\{E\in\mathcal{SH}(pt)|\pi_n(E)_m=0, \forall n\geq 0,m\in\mathbb{Z}\},\]
\[\mathcal{SH}_{\geq 0}=\{E\in\mathcal{SH}(pt)|\pi_n(E)_m=0, \forall n<0,m\in\mathbb{Z}\}.\]
They give a \(t\)-structure (see \cite[5.2]{Mo1}) on \(\mathcal{SH}(pt)\) where \(\mathcal{SH}_{\geq 0}\) can be described as the smallest full subcategory of \(\mathcal{SH}(pt)\) being stable under suspension, extensions and direct sums, containing \(\Sigma^{\infty}X_+\wedge\mathbb{G}^{\wedge i}_m\) for every \(X\in Sm/k\) and \(i\in\mathbb{Z}\) (see \cite[Proposition 2.1.70]{A}). Its heart \(\mathcal{SH}^{\heartsuit}\) is equivalent to the category of homotopy modules (see \cite[Definition 5.2.4]{Mo1}), where the equivalence is given by
\[E\longmapsto\pi_0(E)_*.\]

Define \(\mathcal{SH}^{eff}(pt)\) to be smallest triangulated full subcategory on \(\mathcal{SH}(pt)\) containing \(\Sigma^{\infty}X_+\) for every \(X\in Sm/k\). The functor from \(\mathcal{SH}^{eff}(pt)\)
\[E\longmapsto\pi_*(E)_0\]
is conservative (see \cite[Proposition 4]{B}). Define
\[\mathcal{SH}^{eff}_{\leq_e-1}=\{E\in\mathcal{SH}^{eff}(pt)|\pi_n(E)_0=0, \forall n\geq 0\},\]
\[\mathcal{SH}^{eff}_{\geq_e0}=\{E\in\mathcal{SH}^{eff}(pt)|\pi_n(E)_0=0, \forall n<0\}.\]
They give a \(t\)-structure (see \cite[\S 3]{B}) of \(\mathcal{SH}^{eff}(pt)\) where \(\mathcal{SH}^{eff}_{\geq_e0}\) is the smallest full subcategory of \(\mathcal{SH}^{eff}(pt)\) being stable under suspension, extensions, direct sums and containing \(\Sigma^{\infty}X_+\) for every \(X\in Sm/k\). The functor \(\mathcal{SH}^{eff,\heartsuit}\longrightarrow Ab(Sm/k)\) sending \(E\) to \(\pi_0(E)_0\) is conservative (see \cite[Proposition 5]{B}). We further define
\[\mathcal{SH}^{eff}(pt)(n)=\mathcal{SH}^{eff}(pt)\wedge\mathbb{G}_m^{\wedge n}\]
for any \(n\in\mathbb{Z}\), which has a \(t\)-structure obtained by shifting that of \(\mathcal{SH}^{eff}(pt)\) by \(\mathbb{G}_m^{\wedge n}\).

There is an adjunction
\[i_n:\mathcal{SH}^{eff}(pt)(n)\rightleftharpoons\mathcal{SH}(pt):r_n\]
by \cite[Theorem 4.1]{N} where \(i_n\) is the inclusion. The functor \(r_n\) is \(t\)-exact and \(i_n\) is right \(t\)-exact. Moreover, we have a functor
\[i^{\heartsuit}:\mathcal{SH}(pt)^{eff,\heartsuit}\longrightarrow\mathcal{SH}(pt)^{\heartsuit},\]
whose essential image consists of effective homotopy modules. Define
\[f_n=i_n\circ r_n.\]
We have natural isomorphisms \(r_n\circ i_n\cong Id\) and \(f_{n+1}\circ f_n\cong f_{n+1}\), where the latter induces a natural transformation \(f_{n+1}\longrightarrow f_n\) (see \cite[\S 4]{DLORV}). We define \(s_n(E)\) for any \(E\in\mathcal{SH}(pt)\) by the functorial distinguished triangle
\[f_{n+1}(E)\longrightarrow f_n(E)\longrightarrow s_n(E)\longrightarrow f_{n+1}(E)[1].\]
Finally we have
\[f_n(E)\wedge\mathbb{G}_m=f_{n+1}(E\wedge\mathbb{G}_m)\]
by \cite[Lemma 8]{B}.
\begin{remark}\label{shadj}
There are adjunctions
\[\mathcal{SH}(pt)\rightleftharpoons D_{\mathbb{A}^1}(pt)\rightleftharpoons\widetilde{DM}(pt)\]
by \cite[5.2.25 and 5.3.35]{CD1}, \cite[4.6]{J} and \cite[3.3.6.a, \S 3]{BCDFO}. So we obtain an adjunction
\[\gamma^*:\mathcal{SH}(pt)\rightleftharpoons \widetilde{DM}(pt):\gamma_*\]
where \(\gamma^*\) is monoidal and \(\gamma^*\Sigma^{\infty}X_+=\Sigma^{\infty}\mathbb{Z}(X)\) for every \(X\in Sm/k\).
\end{remark}
\section{Mapping Cone of the Hopf Map in MW-Motives}\label{Mapping Cone of the Hopf Map}
\begin{definition}
The morphism
\[\begin{array}{ccc}\mathbb{A}^2\setminus 0&\longrightarrow&\mathbb{P}^1\\(x,y)&\longmapsto&[x:y]\end{array}\]
induces a morphism
\[\mathbb{G}_m\wedge\mathbb{G}_m\wedge S^1\longrightarrow\mathbb{G}_m\wedge S^1.\]
It is the suspension of a (unique) morphism \(\eta\in[\mathbb{G}_m,\mathbbm{1}]\) (resp. \([\mathbb{Z}(1)[1],\mathbb{Z}]_{{K}}\)), which is called the Hopf map.
\end{definition}
\begin{definition}\label{cone}
Define
\[H\widetilde{\mathbb{Z}}=f_0{K}_*^{MW};H_{\mu}\mathbb{Z}=f_0{K}_*^M;H_W\mathbb{Z}=f_0{K}_*^W;H_{\mu}\mathbb{Z}/2=f_0{K}_*^{M/2}\]
as in \cite[remark before Lemma 12]{B}. Moreover, for any \(E\in\mathcal{SH}(pt)\), define \(E/\eta=E\wedge\mathbbm{1}/{\eta}\).
\end{definition}
We have adjunctions
\[\gamma^*:\mathcal{SH}(pt)\rightleftharpoons DM_{{K}}(pt):\gamma_*\]
by Proposition \ref{dmadj} and Remark \ref{shadj}.
\begin{proposition}\label{spec}
We have
\[\gamma_*(\mathbbm{1})=f_0{K}_*\]
if \({K_*}=K^{MW}_*, K^M_*, K^M_*/2\).
\end{proposition}
\begin{proof}
This follows from \cite[Theorem 5.2 and Corollary 5.4, \S 7]{BCDFO} and \cite[Lemma 12]{B}.
\end{proof}
\begin{proposition}
We have
\begin{enumerate}
\item\[\begin{array}{cc}	
\pi_i(H\widetilde{\mathbb{Z}})_0=\left\{\begin{array}{cc}{K}_0^{MW}&i=0\\0&i\neq 0\end{array}\right.&
\pi_i(H\widetilde{\mathbb{Z}})_1=\left\{\begin{array}{cc}{K}_1^{MW}&i=0\\0&i\neq 0\end{array}\right.
\end{array};\]
\item\[\begin{array}{cc}	
\pi_i(H_{\mu}\mathbb{Z})_0=\left\{\begin{array}{cc}\mathbb{Z}&i=0\\0&i\neq 0\end{array}\right.&
\pi_i(H_{\mu}\mathbb{Z})_1=\left\{\begin{array}{cc}{K}_1^M&i=0\\0&i\neq 0\end{array}\right.
\end{array};\]
\item\[\begin{array}{cc}	
\pi_i(H_W\mathbb{Z})_0=\left\{\begin{array}{cc}{K}_0^W&i=0\\0&i\neq 0\end{array}\right.&
\pi_i(H_W\mathbb{Z})_1=\left\{\begin{array}{cc}{K}_1^W&i=0\\\mathbb{Z}/2&i=1\\0&i\neq 0,1\end{array}\right.
\end{array};\]
\item\[\begin{array}{cc}	
\pi_i(H_{\mu}\mathbb{Z}/2)_0=\left\{\begin{array}{cc}\mathbb{Z}/2&i=0\\0&i\neq 0\end{array}\right.&
\pi_i(H_{\mu}\mathbb{Z}/2)_1=\left\{\begin{array}{cc}K_1^M/2&i=0\\\mathbb{Z}/2&i=1\\0&i\neq 0,1\end{array}\right.
\end{array}.\]
\end{enumerate}
\end{proposition}
\begin{proof}
The calculation of \(\pi_i(-)_0\) is easy since all these spectra are in \(SH^{eff,\heartsuit}\). Let us compute \(\pi_i(-)_1\).
\begin{enumerate}
\item[(2)]This follows from the computation of motivic cohomology groups \(H^{p,1}(X,\mathbb{Z})\) for any \(X\in Sm/k\).
\item[(4)]This follows from the computation of motivic cohomology groups \(H^{p,1}(X,\mathbb{Z}/2)\) for any \(X\in Sm/k\).
\item[(1)]This follows from \cite[Theorem 1.4.1, \S 5]{BCDFO}, \cite[Theorem 17]{B} and (2).
\item[(3)]The case \(i\neq 0\) follows from \cite[Theorem 17]{B} and (4). If \(i=0\), the result follows from \cite[Lemma 19]{B}.
\end{enumerate}
\end{proof}
\begin{proposition}\label{W}
Suppose \(E\in\mathcal{SH}(pt)^{\heartsuit}\) is an effective homotopy module. For any \(X\in Sm/k\) and \(n,m\in\mathbb{Z}\), we have
\[[\Sigma^{\infty}X_+,f_0(E)(n)[m]]=H^{m-n}(X,\pi_0(E)_n)\]
if \(m\geq 2n-1\). In particular, we could pick \(E=K_{*+t}^W\).
\end{proposition}
\begin{proof}
We have the Postnikov spectral sequence (see \cite[Remark 4.3.9]{Mo1})
\[H^p(X,\pi_{-q}(f_0(E))_n)\Longrightarrow[\Sigma^{\infty}X_+,f_0(E)(n)[n+p+q]].\]
The \(r(E)\) (resp. \(f_0(E)\)) is an element in \(\mathcal{SH}(pt)^{eff,\heartsuit}\) (resp. \(\mathcal{SH}(pt)_{\geq 0}\)) by exactness of the functors \(i\) and \(r\). So we have \(\pi_{-q}(f_0(E))_n=0\) if \(q>0\). Furthermore,
\[H^p(X,\pi_{-q}(f_0(E))_n)=0\]
if \(p\geq n\) and \(q<0\) by using \(\pi_{-q}(f_0(E))_0=0\) and the Gersten complex of the homotopy module \(\pi_{-q}(f_0(E))_*\). Finally, we have
\[\pi_0(f_0(E))_*=\pi_0(E)_*\]
by the same arguments as in \cite[Lemma 18]{B}. So the first statement follows.

The second statement follows from \cite[Lemma 6]{B}.
\end{proof}
\begin{proposition}\label{trivial}
We have
\[\begin{array}{cc}H_{\mu}\mathbb{Z}/\eta=H_{\mu}\mathbb{Z}\oplus H_{\mu}\mathbb{Z}\wedge\mathbb{P}^1,&
(H_{\mu}\mathbb{Z}/2)/\eta=H_{\mu}\mathbb{Z}/2\oplus H_{\mu}\mathbb{Z}/2\wedge\mathbb{P}^1\end{array}.\]
\end{proposition}
\begin{proof}
We have \(f_1(H_{\mu}\mathbb{Z})=f_1(H_{\mu}\mathbb{Z}/2)=0\) hence
\[\begin{array}{cc}f_1(H_{\mu}\mathbb{Z}/\eta)\cong H_{\mu}\mathbb{Z}\wedge\mathbb{P}^1,&f_1((H_{\mu}\mathbb{Z}/2)/\eta)\cong H_{\mu}\mathbb{Z}/2\wedge\mathbb{P}^1\end{array}\]
by applying \(f_1\) to the distinguished triangle of \(\mathbbm{1}/{\eta}\). Then the statements follow from the commutative diagrams
\[\begin{array}{cc}
	\xymatrix
	{
		H_{\mu}\mathbb{Z}\ar[r]\ar[d]_{\cong}	&H_{\mu}\mathbb{Z}/\eta\ar[d]\\
		s_0(H_{\mu}\mathbb{Z})\ar[r]^-{\cong}	&s_0(H_{\mu}\mathbb{Z}/\eta)
	}&
	\xymatrix
	{
		H_{\mu}\mathbb{Z}/2\ar[r]\ar[d]_{\cong}	&(H_{\mu}\mathbb{Z}/2)/\eta\ar[d]\\
		s_0(H_{\mu}\mathbb{Z}/2)\ar[r]^-{\cong}	&s_0((H_{\mu}\mathbb{Z}/2)/\eta)
	}.
\end{array}
\]
\end{proof}
\begin{proposition}
We have
\begin{enumerate}
\item\[\pi_i(H\widetilde{\mathbb{Z}}/\eta)_0=\left\{\begin{array}{cc}\mathbb{Z}&i=0\\2{K}_1^M&i=1\\0&i\neq 0,1\end{array}\right..\]
\item\[\pi_i(H_{\mu}\mathbb{Z}/\eta)_0=\left\{\begin{array}{cc}\mathbb{Z}&i=0\\{K}_1^M&i=1\\0&i\neq 0,1\end{array}\right..\]
\item\[\pi_i(H_W\mathbb{Z}/\eta)_0=\left\{\begin{array}{cc}\mathbb{Z}/2&i=0,2\\0&i\neq 0,2\end{array}\right..\]
\item\[\pi_i((H_{\mu}\mathbb{Z}/2)/\eta)_0=\left\{\begin{array}{cc}\mathbb{Z}/2&i=0,2\\K_1^M/2&i=1\\0&i\neq 0,1,2\end{array}\right..\]
\end{enumerate}
\end{proposition}
\begin{proof}
The (2) and (4) follow from Proposition \ref{trivial} and (3) follows from \cite[Lemma 20]{B}. For (1) we apply the long exact sequence
\[\pi_i(E)_1\xrightarrow{\eta}\pi_i(E)_0\longrightarrow\pi_i(E/\eta)_0\longrightarrow\pi_{i-1}(E)_1\xrightarrow{\eta}\pi_{i-1}(E)_0\]
for any spectrum \(E\).
\end{proof}
\begin{proposition}\label{s0}
For any spectrum \(E\in SH^{eff}(pt)\), we have
\[s_0(E/\eta)=s_0(E).\]
\end{proposition}
\begin{proof}
This is because
\[s_0(E\wedge\mathbb{G}_m)=s_{-1}(E)\wedge\mathbb{G}_m=0\]
and \(s_0\) is an exact functor.
\end{proof}
\begin{proposition}\label{slice}
We have \(f_1(H\widetilde{\mathbb{Z}}/\eta)=\mathbb{P}^1\wedge H_{\mu}\mathbb{Z}\)
hence a distinguished triangle
\begin{equation}\label{a}\mathbb{P}^1\wedge H_{\mu}\mathbb{Z}\longrightarrow H\widetilde{\mathbb{Z}}/\eta\longrightarrow H_{\mu}\mathbb{Z}\oplus H_{\mu}\mathbb{Z}/2[2]\longrightarrow \mathbb{P}^1\wedge H_{\mu}\mathbb{Z}[1].\end{equation}
\end{proposition}
\begin{proof}
We have a distinguished triangle
\[H_{\mu}\mathbb{Z}\xrightarrow{h}H\widetilde{\mathbb{Z}}\longrightarrow H_W\mathbb{Z}\longrightarrow H_{\mu}\mathbb{Z}[1].\]
Hence the first statement follows from applying \(f_1\circ(-/\eta)\) on this triangle, \cite[Lemma 20]{B} and Proposition \ref{trivial}. The second statement follows from distinguished triangle
\[f_1(H\widetilde{\mathbb{Z}}/\eta)\longrightarrow f_0(H\widetilde{\mathbb{Z}}/\eta)\longrightarrow s_0(H\widetilde{\mathbb{Z}}/\eta)\longrightarrow f_1(H\widetilde{\mathbb{Z}}/\eta)[1],\]
\cite[Proposition 23]{B} and Proposition \ref{s0},
\end{proof}
\begin{remark}\label{non-split}
The triangle (\ref{a}) above does not split since by applying \(\pi_2(-)_0\) we obtain a non-split exact sequence
\[0\longrightarrow\mathbb{Z}/2\mathbb{Z}\longrightarrow{K}_1^M\longrightarrow2{K}_1^M\longrightarrow0.\]
\end{remark}
\begin{definition}
Suppose \(X\in Sm/k\), \(n\in\mathbb{N}\). We define
\[\eta_K^n(X,R)=Hom_{DM_K^{eff}(pt,R)}(R(X),R/{\eta}(n)[2n]).\]
If \(R=\mathbb{Z}\) and \(K_*=K_*^{MW}\), we have
\[\eta_K^n(X):=\eta_K^n(X,R)=Hom_{SH(pt)}(\Sigma^{\infty}X_+, H\widetilde{\mathbb{Z}}/\eta(n)[2n]).\]
\end{definition}
\begin{proposition}\label{triangles}
The functorial distinguished triangle \(f_1\longrightarrow f_0\longrightarrow s_0\longrightarrow f_1[1]\) induces commutative diagrams of distinguished triangles (rows and columns)
\[
	\xymatrix
	{
		H_{\mu}\mathbb{Z}\wedge\mathbb{P}^1\ar[r]\ar@{=}[d]	&H_{\mu}\mathbb{Z}/\eta\ar[r]\ar[d]_{h}		&H_{\mu}\mathbb{Z}\ar[r]\ar[d]_{a=2i_1}													&\\
		H_{\mu}\mathbb{Z}\wedge\mathbb{P}^1\ar[r]\ar[d]			&H\widetilde{\mathbb{Z}}/\eta\ar[r]\ar[d]	&H_{\mu}\mathbb{Z}\oplus H_{\mu}\mathbb{Z}/2[2]\ar[r]\ar[d]_{b=\pi\oplus id}	&\\
		0\ar[r]\ar[d]																&H_W\mathbb{Z}/\eta\ar[r]^-{\cong}\ar[d]	&H_{\mu}\mathbb{Z}/2\oplus H_{\mu}\mathbb{Z}/2[2]\ar[r]\ar[d]						&\\
																						&																&																														&
	}
\]
and
\[
	\xymatrix
	{
		H_{\mu}\mathbb{Z}\wedge\mathbb{P}^1\ar[r]\ar[d]_{d=2}	&H\widetilde{\mathbb{Z}}/\eta\ar[r]\ar[d]_p	&H_{\mu}\mathbb{Z}\oplus H_{\mu}\mathbb{Z}/2[2]\ar[r]\ar[d]_{e=p_1}\ar[r]	&\\
		H_{\mu}\mathbb{Z}\wedge\mathbb{P}^1\ar[r]\ar[d]		&H_{\mu}\mathbb{Z}/\eta\ar[r]\ar[d]			&H_{\mu}\mathbb{Z}\ar[r]\ar[d]_{0}																			&\\
		H_{\mu}\mathbb{Z}/2\wedge\mathbb{P}^1\ar[r]\ar[d]	&C(\tau[1])\ar[r]\ar[d]								&H_{\mu}\mathbb{Z}/2[3]\ar[r]\ar[d]															&\\
																					&																&																														&
	}
\]
where \(p\circ h=2\cdot id\), \(\tau:H_{\mu}\mathbb{Z}/2[1]\longrightarrow H_{\mu}\mathbb{Z}/2\wedge\mathbb{G}_m\) is the unique nonzero map (see \cite[5.1]{HKO}) and \(\pi\) is the modulo \(2\) map.
\end{proposition}
\begin{proof}
The middle column of the first diagram is the distinguished triangle induced by hyperbolic map. The map \(a\) is equal to the composite (see \cite[Proposition 23]{B})
\[\xymatrix{H_{\mu}\mathbb{Z}\ar[r]^-{\cong}&s_0(H_{\mu}\mathbb{Z})\ar[r]^{s_0(h)}&s_0(H\widetilde{\mathbb{Z}})\ar[r]^-{(u,v)}_-{\cong}&H_{\mu}\mathbb{Z}\oplus H_{\mu}\mathbb{Z}/2[2]}\]
where \(u\) is just the map
\[s_0(H\widetilde{\mathbb{Z}})\longrightarrow s_0(H_{\mu}\mathbb{Z})\cong H_{\mu}\mathbb{Z}\]
and \(v\) is (the minus of) the composite (see \cite[Lemma 20]{B})
\[s_0(H\widetilde{\mathbb{Z}})\longrightarrow s_0(H_W\mathbb{Z})\longrightarrow H_{\mu}\mathbb{Z}/2[2].\]
Hence it is clear that \(a=2i_1\) and \(e=p_1\). Now the distinguished triangles of \(_{\leq_e0},_{\geq_e1}\) of \(H_W\mathbb{Z}/\eta\) and \((H_{\mu}\mathbb{Z}/2)/\eta\) give us a commutative diagram of distinguished triangles (rows and columns)
\begin{equation}\label{eq}
	\xymatrix
	{
		H_{\mu}\mathbb{Z}/2[2]\ar[r]\ar[d]_{q_1=\tau[1]}	&H_W\mathbb{Z}/\eta\ar[r]\ar[d]_q			&H_{\mu}\mathbb{Z}/2\ar[r]\ar@{=}[d]	&\\
		H_{\mu}\mathbb{Z}/2\wedge\mathbb{P}^1\ar[r]\ar[d]	&(H_{\mu}\mathbb{Z}/2)/\eta\ar[r]\ar[d]	&H_{\mu}\mathbb{Z}/2\ar[r]\ar[d]		&\\
		C(\tau[1])\ar@{=}[r]\ar[d]									&C(\tau[1])\ar[r]\ar[d]						&0\ar[r]\ar[d]									&\\
																					&															&														&
	}.
\end{equation}
The \(q_1=\tau[1]\) follows from the proof of \cite[Lemma 20]{B}. Hence we see that \(b=\pi\oplus id\) and that the middle column of the second diagram since \(C(p)=C(q)\). Moreover, \(f_1(H_{\mu}\mathbb{Z}/2)=s_0(H_{\mu}\mathbb{Z}/2\wedge\mathbb{G}_m)=0\) hence \(f_1(C(\tau))=H_{\mu}\mathbb{Z}/2\wedge\mathbb{G}_m\) and \(s_0(C(\tau))=H_{\mu}\mathbb{Z}/2[2]\). Hence the second diagram is derived. Finally, the \(d=2\) since \(p\circ h=2\cdot id\).
\end{proof}
\begin{theorem}\label{eta}
Suppose \(X\in Sm/k\). There are commutative diagrams with exact rows and columns
\[
	\xymatrix
	{
		0\ar[r]	&CH^{n+1}(X)\ar[r]\ar@{=}[d]		&CH^{n+1}(X)\oplus CH^n(X)\ar[r]\ar[d]_{h}	&CH^n(X)\ar[r]\ar[d]_{2}&0\\
					&CH^{n+1}(X)\ar[r]^-u					&\eta_{MW}^n(X)\ar[r]^-v\ar[d]						&CH^n(X)\ar[d]_{\pi}\ar[r]&0\\
					&												&CH^n(X)/2\ar@{=}[r]\ar[d]							&CH^n(X)/2\ar[d]			&\\
					&												&0																&0									&
	},
\]
\[
	\xymatrix
	{
					&CH^{n+1}(X)\ar[r]^-u\ar[d]_{2}		&\eta_{MW}^n(X)\ar[r]^-v\ar[d]_{p}									&CH^n(X)\ar@{=}[d]\ar[r]&0\\
		0\ar[r]	&CH^{n+1}(X)\ar[r]\ar[d]_{\pi}	&CH^{n+1}(X)\oplus CH^n(X)\ar[r]\ar[d]_{a=(\pi,Sq^2\circ\pi)}	&CH^n(X)\ar[r]				&0\\
					&CH^{n+1}(X)/2\ar@{=}[r]\ar[d]	&CH^{n+1}(X)/2\ar[d]													&									&\\
					&0												&0																				&									&
	}.
\]
where \(p\circ h=2\cdot id\) and \(\pi\) is the modulo by \(2\). Consequently, if \(_2CH^{n+1}(X)=0\), we have a Cartesian square
\[
	\xymatrix
	{
		\eta^n_{MW}(X)\ar[r]\ar[d]		&CH^{n+1}(X)\ar[d]_{\pi}\\
		CH^n(X)\ar[r]^-{Sq^2\circ\pi}	&CH^{n+1}(X)/2
	}
\]
being natural with respect to \(X\in Sm/k\) satisfying \(_2CH^{n+1}(X)=0\).
\end{theorem}
\begin{proof}
The diagrams directly follow from Proposition \ref{triangles}. Now we want to study the composite
\[H_{\mu}\mathbb{Z}\longrightarrow H_{\mu}\mathbb{Z}/\eta\longrightarrow C(\tau[1])\]
which is equal to
\[H_{\mu}\mathbb{Z}\longrightarrow H_{\mu}\mathbb{Z}/2\longrightarrow (H_{\mu}\mathbb{Z}/2)/\eta\longrightarrow C(\tau[1]).\]
By diagram (\ref{eq}), it suffices to study the composite
\[H_{\mu}\mathbb{Z}/2\longrightarrow H_W\mathbb{Z}/\eta\longrightarrow (H_{\mu}\mathbb{Z}/2)/\eta\longrightarrow H_{\mu}\mathbb{Z}/2\wedge\mathbb{P}^1.\]
It is equal to the composite by naturality
\[H_{\mu}\mathbb{Z}/2\longrightarrow H_W\mathbb{Z}/\eta\xrightarrow{\beta} H_W\mathbb{Z}\wedge\mathbb{P}^1\xrightarrow{\rho} H_{\mu}\mathbb{Z}/2\wedge\mathbb{P}^1\]
where \(\beta\) is the boundary map of \(H_W\mathbb{Z}\otimes\eta\). We have a commutative diagram of distinguished triangles
\[
	\xymatrix
	{
		H_W\mathbb{Z}\ar[r]\ar@{=}[d]	&H_W\mathbb{Z}/\eta\ar[r]^-{\beta}\ar[d]_a	&H_W\mathbb{Z}\wedge\mathbb{P}^1\ar[r]\ar[d]_b	&\\
		H_W\mathbb{Z}\ar[r]				&H_{\mu}\mathbb{Z}/2\ar[r]^-{\beta'}			&f_0(K_{*+1}^W)[1]\ar[r]									&
	}
\]
where the second row comes from applying \(f_0\) on the exact sequence
\[0\longrightarrow K_{*+1}^W\longrightarrow K_*^W\longrightarrow K_*^M/2\longrightarrow 0.\]
Define
\[X(\varphi)_n=[\Sigma^{\infty}X_+,\varphi(n)[2n]].\]
We see that \(X(a)_n\) and \(X(a[1])_n\) are isomorphisms, hence so does \(X(b)_n\) by Five lemma. By Proposition \ref{W}, \(X(\beta')_n\) is the Bockstein map (see \cite[2.3]{HW})
\[CH^n(X)/2\longrightarrow H^{n+1}(X,K^W_{n+1})\]
, hence so does \(X(\beta)_n\). Moreover the map \(X(\rho)_n\) is the reduction map (see \textit{loc. cit.})
\[H^{n+1}(X,K^W_{n+1})\longrightarrow CH^{n+1}(X)/2.\]
Hence we see that \(a=Sq^2\circ\pi\) on \(CH^n(X)\) by \cite[Remark 10.5]{F}. The composite
\[H_{\mu}\mathbb{Z}\wedge\mathbb{P}^1\longrightarrow H_{\mu}\mathbb{Z}/\eta\longrightarrow (H_{\mu}\mathbb{Z}/2)/\eta\longrightarrow C(\tau[1])\]
is equal to the composite
\[H_{\mu}\mathbb{Z}\wedge\mathbb{P}^1\longrightarrow H_{\mu}\mathbb{Z}/2\wedge\mathbb{P}^1\longrightarrow C(\tau[1])\]
so \(a=\pi\) on \(CH^{n+1}(X)\). Hence \(a=(\pi,Sq^2\circ\pi)\).

Now suppose \(_2CH^{n+1}(X)=0\). The statement we want is equivalent to \(p\) is injective. Suppose \(x\in\eta_{MW}^n(X)\), \(p(x)=0\) so \(v(x)=0\). Hence there is an \(y\in CH^{n+1}(X)\) such that \(u(y)=x\). So \(p(u(y))=2y=0\) hence \(y=0\) by \(_2CH^{n+1}(X)=0\). Hence we are done.
\end{proof}
\begin{coro}\label{etaq}
For any \(X\in Sm/k\), we have natural isomorphisms
\[CH^n(X)[\frac{1}{2}]\oplus CH^{n+1}(X)[\frac{1}{2}]\xrightarrow{\cong}\eta_{MW}^n(X)[\frac{1}{2}]\xrightarrow{\cong}\eta_{MW}^n(X,\mathbb{Z}[\frac{1}{2}]).\]
\end{coro}

\section{The Projective Bundle Theorem}\label{The Projective Bundle Theorem}
Recall the Cartesian square of homotopy modules
\[
	\xymatrix
	{
		{K}_*^{MW}\ar[r]^a\ar[d]_c\ar[rd]^{\pi} 	&{K}_*^W\ar[d]_b\\
		{K}_*^M\ar[r]^d 										&{K}_*^M/2
	}
\]
where each arrow \(f\) in the square induces an adjoint pair \((f^*,f_*)\) between triangulated motivic categories by Proposition \ref{dmadj}.
\begin{proposition}\label{sep}
There is a distinguished triangle
\[\mathbb{Z}\longrightarrow a_*\mathbb{Z}\oplus c_*\mathbb{Z}\longrightarrow {\pi}_*\mathbb{Z}\longrightarrow \mathbb{Z}[1]\]
in \(\widetilde{DM}^{eff}(pt)\).
\end{proposition}
\begin{proof}
By the Cartesian square above since for example \(\mathbb{Z}=K_0^{MW}\) and \(c_*\mathbb{Z}=K_0^M\).
\end{proof}
\begin{proposition}\label{rhom}
We have
\[R\underline{Hom}(\mathbb{Z}(1)[1],\mathbb{Z})=0\]
in \(DM_K^{eff}(pt,\mathbb{Z})\) if \(K_*=K^M_*, K^M_*/2\).
\end{proposition}
\begin{proof}
For any \(X\in Sm/k\) and \(i\in\mathbb{Z}\), we have
\[[\mathbb{Z}(X)(1)[1],\mathbb{Z}[i]]_M=0\]
if \(i\neq 0\) by \cite[Corollary 4.2]{MVW}. If \(i=0\), since \(\mathbb{G}_m\) has a rational point, the equality above still holds. Using the universal coefficient theorem, we see that
\[[\mathbb{Z}(X)(1)[1],\mathbb{Z}[i]]_{M/2}=0.\]
The left hand side is just
\[[\mathbb{Z}(X),R\underline{Hom}(\mathbb{Z}(1)[1],\mathbb{Z})[i]]_K.\]
So we are done.
\end{proof}
\begin{proposition}\label{negative}
Suppose \(q\in\mathbb{N}\) and \(p\in\mathbb{Z}\). We have
\[[\mathbb{Z}(q)[p],\mathbb{Z}]_{MW}=\begin{cases} 0 & \text{if $p\neq q$} \\ W(k) & \text{if $p=q>0$}\\GW(k)&\text{if \(p=q=0\)} ,\end{cases}\]
\[[\mathbb{Z}(q)[p],\mathbb{Z}]_W=\begin{cases} 0 & \text{if $p\neq q$} \\ W(k) & \text{if $p=q$} ,\end{cases}\]
\[[\mathbb{Z}(q)[p],\mathbb{Z}]_M=\begin{cases} 0 & \text{if $p\neq 0$ or $q \neq 0$} \\\mathbb{Z}&\text{if \(p=q=0\)}, \end{cases}\]
\[[\mathbb{Z}(q)[p],\mathbb{Z}]_{M/2}=\begin{cases} 0 & \text{if $p\neq 0$ or $q \neq 0$} \\\mathbb{Z}/2\mathbb{Z}&\text{if \(p=q=0\)}. \end{cases}\]
\end{proposition}
\begin{proof}
We have
\[[\mathbb{Z}(q)[p],\mathbb{Z}]_K=\pi_{p-q}(f_0(K_*))_{-q}(pt)\]
by Proposition \ref{spec} if \(K_*=K^{MW}_*, K^M_*, K^M_*/2\). If \(K_*=K^M_*, K^M_*/2\), the statements follow from Proposition \ref{rhom}. For the case \(K_*=K^{MW}_*\), one uses \cite[Theorem 17]{B}. Finally, the case \(K_*=K^W_*\) follows from the other three cases and Proposition \ref{sep}.
\end{proof}
\begin{proposition}\label{ortho}
\begin{enumerate}
\item Suppose \(i,j,a,b\in\mathbb{N}\). We have
\[[\mathbb{Z}/{\eta}^{\otimes a}(i)[2i],\mathbb{Z}/{\eta}^{\otimes b}(j)[2j+1]]_{MW}=\begin{cases}W(k)&i=j+1,a=b=0\\0&\textrm{else}\end{cases}.\]
\item We have
\[\mathbb{Z}/{\eta}^{\otimes 2}=\mathbb{Z}/{\eta}\oplus \mathbb{Z}/{\eta}(1)[2]\]
in \(\widetilde{DM}^{eff}(pt)\).
\end{enumerate}
\end{proposition}
\begin{proof}
\begin{enumerate}
\item Set
\[A_{i,j,a,b}=[\mathbb{Z}/{\eta}^{\otimes a}(i)[2i],\mathbb{Z}/{\eta}^{\otimes b}(j)[2j+1]]_{MW}.\]

Suppose \(a=b=0\). If \(i<j\), the group vanishes by cancellation and \(H^{p,q}_{MW}(pt,\mathbb{Z})=0\) if \(p>q\), which is a consequence of hypercohomology spectral sequence. If \(i\geq j\), the statement follows from Proposition \ref{negative}.

Now suppose \(a=0, b=1\), we have a long exact sequence
\[[\mathbb{Z}(i)[2i],\mathbb{Z}(j+1)[2j+2]]_{MW}\xrightarrow{\eta}A_{i,j,0,0}\to A_{i,j,0,1}\to A_{i,j+1,0,0}\xrightarrow{\eta[1]}.\]
If \(i-j=2\), the fourth arrow is an isomorphism and \(A_{i,j,0,0}=0\), hence \(A_{i,j,0,1}=0\). If \(i-j=1\), the first arrow is surjective and \(A_{i,j+1,0,0}=0\), hence \(A_{i,j,0,1}=0\). Otherwise \(A_{i,j,0,0}=A_{i,j+1,0,0}=0\) hence \(A_{i,j,0,1}=0\).

Suppose \(a=1, b=0\), we have a long exact sequence
\[[\mathbb{Z}(i)[2i],\mathbb{Z}(j)[2j]]_{MW}\xrightarrow{\eta}A_{i+1,j,0,0}\to A_{i,j,1,0}\to A_{i,j,0,0}\xrightarrow{\eta[1]}.\]
If \(i=j\), \(A_{i,j,0,0}=0\) and the first arrow is surjective. Hence \(A_{i,j,1,0}=0\). If \(i-j=1\), \(A_{i+1,j,0,0}=0\) and the last arrow is an isomorphism. Hence \(A_{i,j,1,0}=0\). Otherwise \(A_{i+1,j,0,0}=A_{i,j,0,0}=0\) so \(A_{i,j,1,0}=0\).

Then the other cases follows trivially from the long exact sequence of \(\eta\) and induction on the pair \((a,b)\).
\item The long exact sequence of \(\eta\) splits after tensoring with \(\mathbb{Z}/{\eta}\) by (1).
\end{enumerate}
\end{proof}
%
%
\begin{theorem}\label{splitting}
\begin{enumerate}
\item If \(n\) is odd, the canonical imbedding \(i:\mathbb{P}^{n-1}\longrightarrow\mathbb{P}^n\) splits in \(\widetilde{DM}^{eff}(pt)\) and we have
\[\mathbb{Z}(\mathbb{P}^n)\cong\mathbb{Z}\oplus\bigoplus_{i=1}^{\frac{n-1}{2}}\mathbb{Z}/{\eta}(2i-1)[4i-2]\oplus\mathbb{Z}(n)[2n]\]
in \(\widetilde{DM}^{eff}(pt)\).
\item If \(n\) is even, the canonical imbedding \(j:\mathbb{P}^{n-2}\longrightarrow\mathbb{P}^n\) splits in \(\widetilde{DM}^{eff}(pt)\) and we have
\[\mathbb{Z}(\mathbb{P}^n)\cong\mathbb{Z}\oplus\bigoplus_{i=1}^{\frac{n}{2}}\mathbb{Z}/{\eta}(2i-1)[4i-2]\]
in \(\widetilde{DM}^{eff}(pt)\)
\end{enumerate}
\end{theorem}
\begin{proof}
Let us prove the statements by induction on \(n\). The cases for \(n=0,1\) are trivial. The case for \(n=2\) follows from \cite[Lemma 6.2.1]{Mo1} and Remark \ref{shadj}. So suppose \(n>2\).
\begin{enumerate}
\item We have a commutative diagram
\[
	\xymatrix
	{
		\mathbb{P}^{n-1}\ar[r]^-z\ar[d]_i	&\mathbb{P}^n\setminus(0:\cdots:0:1)\ar[ld]\\
		\mathbb{P}^{n}								&
	}
\]
where \(z\) is the zero section of \(O_{\mathbb{P}^{n-1}}(1)\). So we have a distinguished triangle
\[\mathbb{Z}(\mathbb{P}^{n-1})\xrightarrow{i}\mathbb{Z}(\mathbb{P}^n)\longrightarrow\mathbb{Z}(n)[2n]\xrightarrow{\partial}\cdots[1]\]
by \cite[Theorem 6.3]{Y}. Now we use Proposition \ref{ortho} and induction hypothesis to conclude that \(\partial=0\).
\item We have a commutative diagram
\[
	\xymatrix
	{
		\mathbb{P}^{n-2}\ar[r]^-z\ar[d]_j	&\mathbb{P}^n\setminus\mathbb{P}^1\ar[ld]\\
		\mathbb{P}^{n}								&
	}
\]
where \(z\) is the zero section of \(O_{\mathbb{P}^{n-2}}(1)^{\oplus 2}\). Furthermore, we have \(det(N_{\mathbb{P}^1/\mathbb{P}^n})=O_{\mathbb{P}^1}(n-1)\). So we have a distinguished triangle by \cite[Theorem 6.3]{Y}
\[\mathbb{Z}(\mathbb{P}^{n-2})\xrightarrow{j}\mathbb{Z}(\mathbb{P}^n)\longrightarrow Th(O_{\mathbb{P}^1}(1))(n-2)[2n-4]\xrightarrow{\partial}\cdots[1].\]
Setting \(n=2\) in the above triangle we find that \(Th(O_{\mathbb{P}^1}(1))=\mathbb{Z}/\eta(1)[2]\). Now we use Proposition \ref{ortho} and induction hypothesis to conclude that \(\partial=0\).
\end{enumerate}
\end{proof}

Recall from Remark \ref{shadj} that we have adjunctions
\[\begin{array}{cc}(\gamma^*,\gamma_*):SH(pt)\leftrightharpoons\widetilde{DM}(pt),&(\gamma'^*,\gamma'_*):\widetilde{DM}(pt)\leftrightharpoons DM(pt).\end{array}\]
\begin{proposition}\label{orthogonality}
Suppose \(i,j\in\mathbb{N}\). We have
\[[\mathbb{Z}(i)[2i],\mathbb{Z}(j)[2j]]_{MW}=\begin{cases}0&\text{if \(i\neq j\)}\\GW(k)&\text{if \(i=j\)},\end{cases}\]
\[[\mathbb{Z}/{\eta}(i)[2i],\mathbb{Z}(j)[2j]]_{MW}=\begin{cases}0&\text{if \(j-i>1\) or \(j<i\)}\\\mathbb{Z}&\text{if \(j-i=1\)}\\2\mathbb{Z}&\text{if \(j=i\)},\end{cases}\]
\[[\mathbb{Z}(i)[2i],\mathbb{Z}/{\eta}(j)[2j]]_{MW}=\begin{cases}0&\text{if \(i-j>1\) or \(i<j\)}\\\mathbb{Z}&\text{if \(i=j\)}\\2\mathbb{Z}&\text{if \(i-j=1\)},\end{cases}\]
\[[\mathbb{Z}/{\eta}(i)[2i],\mathbb{Z}/{\eta}(j)[2j]]_{MW}=\begin{cases}0&\text{if \(|i-j|>1\)}\\2\mathbb{Z}&\text{if \(j=i-1\)}\\\mathbb{Z}\oplus 2\mathbb{Z}&\text{if \(j=i\)}\\\mathbb{Z}&\text{if \(j=i+1\)}\end{cases}.\]
\end{proposition}
\begin{proof}
We set
\[A_{i,j,a,b}=[\mathbb{Z}/{\eta}^{\otimes a}(i)[2i],\mathbb{Z}/{\eta}^{\otimes b}(j)[2j]]_{MW}.\]

Suppose \(a=b=0\). If \(i<j\), the group vanishes by cancellation and \(H^{p,q}_{MW}(pt,\mathbb{Z})=0\) if \(p>q\). If \(i\geq j\), the statement follows from Proposition \ref{negative}.

Suppose \(a=1, b=0\), we have a long exact sequence
\[[\mathbb{Z}(i)[2i+1],\mathbb{Z}(j)[2j]]_{MW}\xrightarrow{\eta}A_{i+1,j,0,0}\to A_{i,j,1,0}\to A_{i,j,0,0}\xrightarrow{\eta[1]}.\]
If \(i=j\), \(A_{i+1,j,0,0}=0\) and the fourth arrow is \(K_0^{MW}(k)\longrightarrow W(k)\), hence \(A_{i,j,1,0}=2\mathbb{Z}\). If \(j-i=1\), \(A_{i,j,0,0}=0\) and the first arrow is \(K_1^{MW}(k)\xrightarrow{\eta}K_0^{MW}(k)\), hence \(A_{i,j,1,0}=\mathbb{Z}\). Otherwise \(A_{i+1,j,0,0}=A_{i,j,0,0}=0\) so \(A_{i,j,1,0}=0\).

Suppose \(a=0, b=1\), we have a long exact sequence
\[[\mathbb{Z}(i)[2i],\mathbb{Z}(j+1)[2j+1]]_{MW}\xrightarrow{\eta}A_{i,j,0,0}\to A_{i,j,0,1}\to A_{i,j+1,0,0}\xrightarrow{\eta[1]}.\]
If \(i=j\), the first arrow is \(K_1^{MW}(k)\xrightarrow{\eta}K_0^{MW}(k)\) and \(A_{i,j+1,0,0}=0\), hence \(A_{i,j,0,1}=\mathbb{Z}\). If \(i-j=1\), \(A_{i,j,0,0}=0\) and the fourth arrow is \(K_0^{MW}(k)\longrightarrow W(k)\), hence \(A_{i,j,0,1}=2\mathbb{Z}\). Otherwise \(A_{i,j,0,0}=A_{i,j+1,0,0}=0\) hence \(A_{i,j,0,1}=0\).

Suppose \(a=1, b=1\), we have a long exact sequence
\[[\mathbb{Z}/\eta(i)[2i],\mathbb{Z}(j+1)[2j+1]]_{MW}\xrightarrow{\eta}A_{i,j,1,0}\to A_{i,j,1,1}\to A_{i,j+1,1,0}\xrightarrow{\eta[1]}.\]
If \(j=i-1\), \(A_{i,j,1,0}=0\) and the fourth arrow is zero by Proposition \ref{ortho}. So \(A_{i,j,1,1}=A_{i,j+1,1,0}=2\mathbb{Z}\). If \(j=i+1\), \(A_{i,j+1,1,0}=0\) and the first term is zero since it is the cokernel of \(Id_{W(k)}\) by again using an exact sequence of \(\eta\). Hence \(A_{i,j,1,0}=A_{i,j,1,1}=\mathbb{Z}\). If \(i=j\), the second arrow is injective because we have the commutative diagram
\[
	\xymatrix
	{
		[\mathbb{Z}/\eta,\mathbb{Z}]_{MW}\ar[r]^-f\ar[d]_{\gamma'^*}^2	&[\mathbb{Z}/\eta,\mathbb{Z}/\eta]_{MW}\ar[d]^{\gamma'^*}\\
		[\mathbb{Z}/\eta,\mathbb{Z}]_M\ar[r]^-{g}							&[\mathbb{Z}/\eta,\mathbb{Z}/\eta]_M
	}
\]
where the \(g\) is injective, hence so does \(f\). The fourth arrow is zero by Proposition \ref{ortho}. Hence \(A_{i,j,1,1}=\mathbb{Z}\oplus2\mathbb{Z}\). Otherwise \(A_{i,j,1,1}\) clearly vanishes.
\end{proof}
\begin{coro}
(see \cite[Corollary 11.8]{F}) Suppose \(0\leq i\leq n\), we have
\[\widetilde{CH}^{i}(\mathbb{P}^n)=\begin{cases}GW(k)&\text{if \(i=0\) or \(i=n\) and \(n\) is odd}\\\mathbb{Z}&\text{if \(i>0\) is even}\\2\mathbb{Z}&\text{else}.\end{cases}\]
\end{coro}
\begin{proposition}\label{left}
Suppose \(X\in Sm/k\) and \(i,j\in\mathbb{N}\), we have a natural isomorphism
\[[\mathbb{Z}(X)/\eta(i)[2i],\mathbb{Z}(j)[2j]]_{MW}\cong\left\{\begin{array}{cc}\eta^{j-i-1}_{MW}(X)&j>i\\0&j<i\\2CH^0(X)&j=i\end{array}\right..\]
\end{proposition}
\begin{proof}
The case \(j<i\) follows from \cite[Lemma 2.26]{Y}. Let us suppose \(j\geq i\) and work in \(\widetilde{DM}(pt)\). The \(\Sigma^{\infty}\mathbb{Z}(\mathbb{P}^2)\) has a strong dual \(\Sigma^{\infty}Th(O_{\mathbb{P}^2}(1))(-3)[-6]\) by \cite[Theorem 6.3]{Y} and \cite[Proposition 2.4.31]{CD1}, which is isomorphic to \(\Sigma^{\infty}\mathbb{Z}/\eta(-2)[-4]\oplus\Sigma^{\infty}\mathbb{Z}\) by Theorem \ref{splitting}. Hence the strong dual of \(\Sigma^{\infty}\mathbb{Z}/\eta\) is \(\Sigma^{\infty}\mathbb{Z}/\eta(-1)[-2]\). Hence we obtain the result when \(j>i\). If \(j=i\), we are going to consider the group
\[[\mathbb{Z}(X)(1)[2],\mathbb{Z}/\eta]_{MW}.\]
We have a long exact sequence induced by \(\eta\)
\[0\longrightarrow[\mathbb{Z}(X)(1)[2],\mathbb{Z}/\eta]_{MW}\longrightarrow[\mathbb{Z}(X),\mathbb{Z}]_{MW}\xrightarrow{X\otimes\eta}[\mathbb{Z}(X)(1)[1],\mathbb{Z}]_{MW}\]
where the last arrow is identified with \(GW(X)\longrightarrow W(X)\) by \cite[Lemma 3.1.8, \S 2]{BCDFO} and the second arrow is injective by \cite[Lemma 2.26]{Y}. So the statement follows.
\end{proof}

Note that we have \(Sq^2(s^i)=s^{i+1}\) for \(i\) being odd, \(s\in CH^1(X)/2\) and \(X\in Sm/k\) by \cite[Lemma 2.10]{HW}.
\begin{proposition}\label{gen}
\begin{enumerate}
\item We have a ring isomorphism
\[End_{\widetilde{DM}^{eff}(pt)}(\mathbb{Z}/{\eta})\cong\mathbb{Z}[x]/(x^2-2x).\]
\item Suppose \(j=2i-1\leq n-1\). The morphism
\[\begin{array}{cccc}\gamma'^*:&[\mathbb{Z}(\mathbb{P}^n),\mathbb{Z}/{\eta}(j)[2j]]_{MW}&\longrightarrow&[\mathbb{Z}(\mathbb{P}^n),\mathbb{Z}/{\eta}(j)[2j]]_M\\&&&=CH^j(\mathbb{P}^n)\oplus CH^{j+1}(\mathbb{P}^n)\end{array}\]
is injective with \(coker(\gamma'^*)=\mathbb{Z}/2\mathbb{Z}\). Moreover, we have
\[(c_1(O_{\mathbb{P}^n}(1))^j,c_1(O_{\mathbb{P}^n}(1))^{j+1})\in Im(\gamma'^*).\]
\end{enumerate}
\end{proposition}
\begin{proof}
Let
\[\begin{array}{cc}R_M=End_{DM^{eff}(pt)}(\mathbb{Z}/{\eta})&R_{MW}=End_{\widetilde{DM}^{eff}(pt)}(\mathbb{Z}/{\eta});\\F_M=[\mathbb{Z}(\mathbb{P}^n),\mathbb{Z}/{\eta}(j)[2j]]_{M}&F_{MW}=[\mathbb{Z}(\mathbb{P}^n),\mathbb{Z}/{\eta}(j)[2j]]_{MW}.\end{array}\]
\begin{enumerate}
\item The long exact sequence of \(\eta\) gives a commutative diagram (*) with exact rows 
\[
	\xymatrixcolsep{11pt}
	\xymatrix
	{
		0\ar[r]	&[\mathbb{Z}/{\eta},\mathbb{Z}]_{MW}\ar[r]^-f\ar[d]_{\gamma'^*}^2	&R_{MW}\ar[r]^-h\ar[d]_{\gamma'^*}	&[\mathbb{Z}/{\eta},\mathbb{Z}(1)[2]]_{MW}\ar[r]\ar[d]^{\cong}_{\gamma'^*}	&0\\
		0\ar[r]	&[\mathbb{Z}/{\eta},\mathbb{Z}]_M\ar[r]															&R_M\ar[r]								&[\mathbb{Z}/{\eta},\mathbb{Z}(1)[2]]_M\ar[r]											&0
	},
\]
which implies that
\[\gamma'^*:R_{MW}\longrightarrow R_M\]
is injective. Here \(f\) is injective as stated in Proposition \ref{orthogonality} and \(h\) is surjective by Proposition \ref{ortho}. Since \(\mathbb{Z}/{\eta}=\mathbb{Z}\oplus\mathbb{Z}(1)[2]\) in \(DM^{eff}(pt)\), we see that
\[R_M\cong\mathbb{Z}\oplus\mathbb{Z}\]
as rings. Hence \(R_{MW}\) is a subring of \(\mathbb{Z}\oplus\mathbb{Z}\). It has a free \(\mathbb{Z}\)-basis \(\{a,b\}\) with \(a=(2,0)\) and it must contain \((1,1)\) so we may assume \(b=(1,1)\). Hence we see that there is a ring isomorphism
\[\begin{array}{ccc}\mathbb{Z}[x]/(x^2-2x)&\longrightarrow&R_{MW}\subseteq\mathbb{Z}\oplus\mathbb{Z}\\x&\longmapsto&(2,0)\end{array}.\]
\item 
We have a commutative diagram with exact rows
\[
	\xymatrix
	{
		0\ar[r]	&\widetilde{CH}^j(\mathbb{P}^n)\ar[r]\ar[d]^2_{\gamma'^*}	&F_{MW}\ar[r]\ar[d]_{\gamma'^*}	&\widetilde{CH}^{j+1}(\mathbb{P}^n)\ar[r]\ar[d]_{\gamma'^*}^{\cong}	&0\\
		0\ar[r]	&CH^j(\mathbb{P}^n)\ar[r]													&F_{M}\ar[r]									&CH^{j+1}(\mathbb{P}^n)\ar[r]						&0
	}
\]
by using Theorem \ref{splitting} and the diagram (*) above. Hence the first statement follows. The second statement follows from Theorem \ref{eta}.
%
\end{enumerate}
\end{proof}
\begin{definition}
Suppose \(n\in\mathbb{N}\) and \(k\leq n-1\) is odd. Define
\[c^k_n=(c_1(O_{\mathbb{P}^n}(1))^k,c_1(O_{\mathbb{P}^n}(1))^{k+1})\in\eta_{MW}^k(\mathbb{P}^n).\]
If \(n\) is odd, define
\[{\tau}_{n+1}=i_*(1)\in\widetilde{CH}^n(\mathbb{P}^n)\]
where \(i:pt\longrightarrow\mathbb{P}^n\) is a rational point (after fixing an isomorphism \(i^*\omega_{\mathbb{P}^n}\cong k\)).
\end{definition}
The \({\tau}_{n+1}\) is independent of the choice of the rational point by \cite[Proposition 6.3]{F}.
\begin{theorem}\label{splitting1}
Suppose \(n\in\mathbb{N}\) and \(p:\mathbb{P}^n\longrightarrow pt\) is the structure map.
\begin{enumerate}
\item If \(n\) is odd, the morphism
\[\Theta:R(\mathbb{P}^n)\longrightarrow R\oplus\bigoplus_{i=1}^{\frac{n-1}{2}}R/{\eta}(2i-1)[4i-2]\oplus R(n)[2n]\]
is an isomorphism in \(\widetilde{DM}^{eff}(pt,R)\) where
\[\Theta=(p,c^{2i-1}_n\otimes R,{\tau}_{n+1}\otimes R).\]
\item If \(n\) is even, the morphism
\[\Theta:R(\mathbb{P}^n)\longrightarrow R\oplus\bigoplus_{i=1}^{\frac{n}{2}}R/{\eta}(2i-1)[4i-2]\]
is an isomorphism in \(\widetilde{DM}^{eff}(pt,R)\) where
\[\Theta=(p,c^{2i-1}_n\otimes R).\]
\end{enumerate}
\end{theorem}
\begin{proof}
It suffices to do the case \(R=\mathbb{Z}\). By \cite[Lemma 5.3]{Y}, Theorem \ref{splitting} and Proposition \ref{orthogonality}, we only have to prove
\begin{enumerate}
\item[(a)] that \(p\) is a free generator of \([\mathbb{Z}(\mathbb{P}^n),\mathbb{Z}]_{MW}\) as an \(End_{\widetilde{DM}^{eff}(pt)}(\mathbb{Z})\)-module;

This is clear since \(p\) corresponds to \(1\in\widetilde{CH}^0(\mathbb{P}^n)\).
\item[(b)] that \(c^{2i-1}_n\) is a free generator of \(\eta^{2i-1}_{MW}(\mathbb{P}^n)\) as an \(End_{\widetilde{DM}^{eff}(pt)}(\mathbb{Z}/{\eta})\)-module if \(2i-1\leq n-1\);

We have a commutative diagram by Proposition \ref{gen}, (2)
\[
	\xymatrix
	{
		End_{\widetilde{DM}^{eff}(pt)}(\mathbb{Z}/{\eta})\ar[r]^{\gamma'^*}\ar[d]_{u=c^n_{2i-1}}&End_{DM^{eff}(pt)}(\mathbb{Z}/{\eta})\ar[d]_{v=c^n_{2i-1}}^{\cong}\\
		\eta_{MW}^{2i-1}(\mathbb{P}^n)\ar[r]^{\gamma'^*}			&\eta_M^{2i-1}(\mathbb{P}^n)
	}
\]
where all terms are isomorphic to \(\mathbb{Z}\oplus\mathbb{Z}\) as abelian groups. The \(v\) is an isomorphism by the free generator's statement in \(DM^{eff}(pt)\). The \(u\) is injective and \(coker(v\circ\gamma'^*)=\mathbb{Z}/2\mathbb{Z}\) by computation in Proposition \ref{gen}, (1). The lower horizontal map is injective but not surjective by Proposition \ref{gen}, (2). Hence we see that \(u\) must be an isomorphism. We are done.
\item[(c)] that \({\tau}_{n+1}\) is a free generator of \([\mathbb{Z}(\mathbb{P}^n),\mathbb{Z}(n)[2n]]_{MW}\) as an \(End_{\widetilde{DM}^{eff}(pt)}(\mathbb{Z}(n)[2n])\)-module if \(n\) is odd;

This is because \(p_*:\widetilde{CH}^n(\mathbb{P}^n)\longrightarrow\widetilde{CH}^0(pt)\) is an isomorphism by the end of the proof of \cite[Proposition 6.3]{F}, so does \(i_*\).
\end{enumerate}
\end{proof}
\begin{definition}
Let \(X\in Sm/k\). An acyclic covering of \(X\) is an open covering \(\{U_i\}\) of \(X\) such that \(CH^j(U_i)=0\) for any \(i\) and \(j>0\).
\end{definition}
Note that any refinement of an acyclic covering is again acyclic. Acyclic covering exists if \(X\) admits an open covering by \(\mathbb{A}^n\).
\begin{proposition}\label{global}
Suppose \(X\in Sm/k\), \(_2CH^*(X)=0\) and \(E\) is a vector bundle of rank \(n\) over \(X\). Assume that \(X\) admits an acyclic covering \(\{U_{\alpha}\}\) such that we have trivializations \(E|_{U_{\alpha}}=O_{U_{\alpha}}^{\oplus n}\). Then for any \(\sigma\in\eta_{MW}^i(\mathbb{P}^{n-1})\) where \(i\in\mathbb{N}\), there is a \(\sigma(E)\in\eta_{MW}^i(\mathbb{P}(E))\) such that
\[\sigma(E)|_{U_{\alpha}}=p^*\sigma\]
for every \(\alpha\) where \(p:\mathbb{P}^{n-1}\times U_{\alpha}\longrightarrow\mathbb{P}^{n-1}\) is the projection.
\end{proposition}
\begin{proof}
By Theorem \ref{eta}, any \(\sigma\in\eta_{MW}^i(\mathbb{P}^{n-1})\) corresponds to an element \[(a\cdot c_1(O(1))^i,b\cdot c_1(O(1))^{i+1})\in CH^i(\mathbb{P}^{n-1})\oplus CH^{i+1}(\mathbb{P}^{n-1})\]
such that
\[\left\{\begin{array}{cc}a\equiv b\ (mod\ 2)&\textrm{if \(i\) is odd}\\b\equiv 0\ (mod\ 2)&\textrm{if \(i\) is even}\end{array}\right..\]
So the element
\[\sigma(E)=(a\cdot c_1(O_E(1))^i,b\cdot c_1(O_E(1))^{i+1})\in\eta_{MW}^i(\mathbb{P}(E))\]
is what we want.
\end{proof}
We denote the quotient map \(\mathbb{Z}\longrightarrow\mathbb{Z}/\eta\) by \(\pi\). By strong duality, it corresponds to a morphism \(\partial:\mathbb{Z}/\eta\longrightarrow\mathbb{Z}(1)[2]\). Since \(\pi\) is a generator of \([\mathbb{Z},\mathbb{Z}/\eta]_{MW}=\mathbb{Z}\) by the computation in Proposition \ref{orthogonality}, we see that \(\partial\) is, up to a sign, equal to the boundary map \(\partial'\) of \(\mathbb{Z}/\eta\), since both \(\partial\) and \(\partial'\) are generators of \([\mathbb{Z}/\eta,\mathbb{Z}(1)[2]]_{MW}=\mathbb{Z}\).

Now we are ready to formulate our projective bundle theorems.
\begin{theorem}\label{pbt}
Let \(S\in Sm/k\), \(X\in Sm/S\) and \(E\) be a vector bundle of rank \(n\) over \(X\). Suppose \(_2CH^*(X)=0\) and \(X\) admits an acyclic covering. Denote by \(p:\mathbb{P}(E)\longrightarrow X\) and \(q:X\longrightarrow S\) the structure maps.
\begin{enumerate}
\item We have isomorphisms
\[R(\mathbb{P}(E))/\eta\cong\bigoplus_{i=0}^{n-1}R(X)/\eta(i)[2i]\]
\[Th(E)/\eta\cong R(X)/\eta(n)[2n]\]
in \(\widetilde{DM}(S,R)\).
\item If \(n\) is odd, the morphism
\[\Theta:R(\mathbb{P}(E))\longrightarrow R(X)\oplus\bigoplus_{i=1}^{\frac{n-1}{2}}R(X)/{\eta}(2i-1)[4i-2]\]
is an isomorphism in \(\widetilde{DM}^{eff}(S,R)\) where
\[\Theta=(p,p\boxtimes c^{2i-1}_n(E)_R).\]
\item If \(n\) is even, we have a distinguished triangle
\begin{footnotesize}
\[Th(det(E)^{\vee})(n-2)[2n-4]\to R(\mathbb{P}(E))\xrightarrow{\Theta} R(X)\oplus\bigoplus_{i=1}^{\frac{n}{2}-1}R(X)/{\eta}(2i-1)[4i-2]\xrightarrow{\mu}\cdots[1]\]
\end{footnotesize}
in \(\widetilde{DM}(S,R)\) where
\[\Theta=(p,p\boxtimes c^{2i-1}_n(E)_R)\]
and \(\mu=0\) if \(e(E^{\vee})=0\in H^n(X,W(det(E)))\).
\end{enumerate}
\end{theorem}
\begin{proof}
It suffices to do the case \(R=\mathbb{Z}\). Let \(\{U_{\alpha}\}\) be a finite acyclic covering of \(X\) such that \(E|_{U_{\alpha}}\) is trivial for any \(i\). Suppose at first \(S=X\), so \(\mathbb{Z}(X)=\mathbb{Z}\). For the general cases, just apply \(q_{\#}\).
\begin{enumerate}
\item We have
\[[\mathbb{Z}(\mathbb{P}(E))/\eta,\mathbb{Z}/\eta(i)[2i]]_{MW}=\eta_{MW}^{i-1}(\mathbb{P}(E))\oplus\eta_{MW}^i(\mathbb{P}(E))\]
by the strong duality of \(\mathbb{Z}/\eta\) for any \(i\in\mathbb{N}\). So by Proposition \ref{global}, for every \(\tau\in[\mathbb{Z}(\mathbb{P}^{n-1})/\eta,\mathbb{Z}/\eta(i)[2i]]_{MW}\), we could find an
\[\tau(E)\in[\mathbb{Z}(\mathbb{P}(E))/\eta,\mathbb{Z}/\eta(i)[2i]]_{MW}\]
such that \(\tau(E)|_{U_{\alpha}}\) is a pullback of \(\tau\) along projection, after fixing a trivialization of \(E|_{U_{\alpha}}\).

Now tensoring the equation in Theorem \ref{splitting1} with \(\mathbb{Z}/\eta\) we obtain an isomorphism
\[(\sigma_a):\mathbb{Z}(\mathbb{P}^{n-1})/\eta\longrightarrow\bigoplus_{i=0}^{n-1}\mathbb{Z}/\eta(i)[2i]\]
by Proposition \ref{ortho}. Then
\[(\sigma_a(E)):\mathbb{Z}(\mathbb{P}(E))/\eta\longrightarrow\bigoplus_{i=0}^{n-1}\mathbb{Z}/\eta(i)[2i]\]
is an isomorphism in \(\widetilde{DM}(X)\) by Proposition \ref{zarsep}.

For the second statement, let us look at the canonical embedding \(j:\mathbb{P}^{n-1}\longrightarrow\mathbb{P}^n\). If \(n\) is odd, for any \(2i-1\leq n-1\), we have
\[c_{2i-1}^n\circ j=c_{2i-1}^{n-1}.\]
If \(n\) is even, the above equation also holds for \(2i-1<n-1\). By Proposition \ref{ortho} and Theorem \ref{splitting1}, there is a morphism \(\pi\in[\mathbb{Z},\mathbb{Z}/\eta]_{MW}\) such that we have a commutative diagram in \(\widetilde{DM}(pt)\) by Proposition \ref{orthogonality}
\[
	\xymatrixcolsep{5pc}
	\xymatrix
	{
		\mathbb{Z}(\mathbb{P}^{n-1})\ar[r]^j\ar[d]_{\tau_n}			&\mathbb{Z}(\mathbb{P}^n)\ar[d]_{c_{n-1}^n}\\
		\mathbb{Z}(n-1)[2n-2]\ar[r]^{\pi(n-1)[2n-2]}	&\mathbb{Z}/\eta(n-1)[2n-2]
	}.
\]
Furthermore, there is a distinguished triangle
\[\mathbb{Z}\xrightarrow{\pi}\mathbb{Z}/\eta\longrightarrow\mathbb{Z}(1)[2]\longrightarrow\cdots[1]\]
since \(C(j)=\mathbb{Z}(n)[2n]\). Applying \([\mathbb{Z},-]_{MW}\) to the triangle above we conclude by Proposition \ref{orthogonality} that \(\pi\) is just the quotient map, that is, a generator of \([\mathbb{Z},\mathbb{Z}/\eta]_{MW}\). So we see that \(\tau_n/\eta\) is equal to the composite
\begin{footnotesize}
\[\mathbb{P}^{n-1}/\eta\xrightarrow{j/\eta}\mathbb{P}^n/\eta\xrightarrow{c_{n-1}^n/\eta}\mathbb{Z}/\eta(n-1)[2n-2]\oplus\mathbb{Z}/\eta(n)[2n]\longrightarrow\mathbb{Z}/\eta(n-1)[2n-2].\]
\end{footnotesize}
So for any \(n\), we have a commutative diagram
\[
	\xymatrix
	{
		\mathbb{Z}(\mathbb{P}^{n-1})/\eta\ar[r]^{j/\eta}\ar[d]_{\cong}^{(\mu_i)}		&\mathbb{Z}(\mathbb{P}^n)/\eta\ar[d]_{\cong}^{(\nu_i)}\\
		\bigoplus_{i=0}^{n-1}\mathbb{Z}/\eta(i)[2i]\ar[r]\ar@{=}[rd]			&\bigoplus_{i=0}^{n}\mathbb{Z}/\eta(i)[2i]\ar[d]_{r}\\
																											&\bigoplus_{i=0}^{n-1}\mathbb{Z}/\eta(i)[2i]
	}
\]
where \(r\) is the projection. Hence there is a commutative diagram in \(\widetilde{DM}(X)\)
\[
	\xymatrix
	{
		\mathbb{Z}(\mathbb{P}(E^{\vee}))/\eta\ar[r]\ar[d]_{\cong}^{(\mu_i(E^{\vee}))}			&\mathbb{Z}(\mathbb{P}(E^{\vee}\oplus O_X))/\eta\ar[d]_{\cong}^{(\nu_i(E^{\vee}\oplus O_X))}	\\
		\bigoplus_{i=0}^{n-1}\mathbb{Z}/\eta(i)[2i]\ar[r]\ar@{=}[rd]	&\bigoplus_{i=0}^{n}\mathbb{Z}/\eta(i)[2i]\ar[d]_r\\
																									&\bigoplus_{i=0}^{n-1}\mathbb{Z}/\eta(i)[2i]
	}
\]
which gives the statement since \(Th(E)=\mathbb{Z}(\mathbb{P}(E^{\vee}\oplus O_X))/\mathbb{Z}(\mathbb{P}(E^{\vee}))\).
\item On each \(U_{\alpha}\), the statement follows from pulling back the equation in Theorem \ref{splitting1}, (1) to \(U_{\alpha}\) and fix an isomorphism \(E|_{U_{\alpha}}=O_{U_{\alpha}}^{\oplus n}\). Then the statement holds on \(X\) by Proposition \ref{zarsep}. Here we have to use \(_2CH^*(U_{\alpha})=0\) to ensure the naturality of \(c^{2i-1}_n(E)\) by Theorem \ref{eta}.
\item We have a commutative diagram in \(\widetilde{DM}(X)\) with rows and columns being distinguished triangles
\[\footnotesize\xymatrixcolsep{0.8pc}
	\xymatrix
	{
		0\ar[r]\ar[d]												&\mathbb{Z}(X)/\eta(n-1)[2n-2]\ar@{=}[r]\ar[d]_{\theta}			&\mathbb{Z}(X)/\eta(n-1)[2n-2]\ar[d]\\
		\mathbb{Z}(\mathbb{P}(E))\ar[r]\ar@{=}[d]	&\mathbb{Z}(\mathbb{P}(E\oplus O_X))\ar[r]_-t\ar[d]_{\Theta}	&Th(det(E)^{\vee})(n-1)[2n-2]\ar[d]\\
		\mathbb{Z}(\mathbb{P}(E))\ar[r]^-{\Theta}	&\mathbb{Z}(X)\oplus\bigoplus_{i=1}^{\frac{n}{2}-1}\mathbb{Z}(X)/{\eta}(2i-1)[4i-2]\ar[r]^-{\mu}	&C
	}
\]
where the middle column follows from (2) and splits. Now we are going to show:
\begin{enumerate}
\item\textit{We have \(C=Th(det(E)^{\vee})(n-2)[2n-3]\)}.

We claim the diagram
\[\xymatrix{Th(det(E)^{\vee})/\eta(n-2)[2n-4]\ar[r]_-{\cong}^-{\varphi(-1)[-2]}\ar[rd]_{Th(det(E)^{\vee})(n-2)[2n-4]\otimes\partial}&\mathbb{Z}(X)/\eta(n-1)[2n-2]\ar[d]_{t\circ{\theta}}\\&Th(det(E)^{\vee})(n-1)[2n-2]}\]
commutes, where the \(\varphi\) is given by (1) and \(\partial\) is the boundary map of \(\mathbb{Z}/\eta\).

To prove this, we adopt the last diagram used in (1), obtaining a commutative diagram whose rows are distinguished triangles
\[\footnotesize\xymatrixcolsep{0.55pc}
	\xymatrix
	{
		\mathbb{Z}(\mathbb{P}(E))/\eta\ar[r]\ar[d]					&\mathbb{Z}(\mathbb{P}(E\oplus O_X))/\eta\ar[r]^-{t/\eta}\ar[d]_{\Theta/\eta}	&Th(det(E^{\vee}))/\eta(n-1)[2n-2]\ar[d]_{\varphi}\\
		\bigoplus_{i=0}^{n-1}\mathbb{Z}(X)/\eta(i)[2i]\ar[r]	&\bigoplus_{i=0}^n\mathbb{Z}(X)/\eta(i)[2i]\ar[r]^{\rho}		&\mathbb{Z}(X)/\eta(n)[2n]
	}.
\]
Suppose
\[\theta':\mathbb{Z}(X)(n)[2n]\longrightarrow\mathbb{Z}(\mathbb{P}(E\oplus O_X))/\eta\]
is the map obtained from \(\theta\) by the strong duality of \(\mathbb{Z}/\eta\). It suffices to show that the composite \(f\)
\begin{footnotesize}
\[
	\xymatrix
	{
		\mathbb{Z}(X)/\eta(n)[2n]\ar[r]^-{\theta'/\eta}				&\mathbb{Z}(\mathbb{P}(E\oplus O_X))\otimes(\mathbb{Z}/\eta)^2\ar[d]_{Id\otimes\mu}\\
		Th(det(E)^{\vee})/\eta(n-1)[2n-2]									&\mathbb{Z}(\mathbb{P}(E\oplus O_X))/\eta\ar[l]_-{t/\eta}
	}
\]
\end{footnotesize}
is the inverse of \(\varphi\) where \(\mu:(\mathbb{Z}/\eta)^2\longrightarrow\mathbb{Z}/\eta\) is given by Proposition \ref{ortho}, (2). By the diagram before, \(\varphi\circ f\) is equal to the composite
\[
	\xymatrix
	{
		\mathbb{Z}(X)/\eta(n)[2n]\ar[r]^-{\theta'/\eta}			&\mathbb{Z}(\mathbb{P}(E\oplus O_X))\otimes(\mathbb{Z}/\eta)^2\ar[d]_{Id\otimes\mu}\\
		\oplus_{i=0}^n\mathbb{Z}(X)/\eta(i)[2i]\ar[d]_{\rho}	&\mathbb{Z}(\mathbb{P}(E\oplus O_X))/\eta\ar[l]_-{\Theta/\eta}\\
		\mathbb{Z}(X)/\eta(n)[2n]											&
	}.
\]
But \({\theta}\) is a section of \(c_n^{n-1}(E)\), so the claim follows from the strong duality of \(\mathbb{Z}/\eta\), which gives the computation of \(C\). Hence we obtain the triangle in the theorem.
\item\textit{We have
\[[\mathbb{Z}(X)/\eta(2i-1)[4i-2],C]=0\]
in \(\widetilde{DM}(X)\) where \(2i-1<n-1\).}

We have an isomorphism
\[[\mathbb{Z}(X)/\eta(2i-1)[4i-2],C]\cong[\mathbb{Z}(X),\mathbb{Z}/\eta(n-2i-1)[2n-4i-1]]\]
in \(\widetilde{DM}(X)\) by (1) and strong duality of \(\mathbb{Z}/\eta\). But the latter is zero by Proposition \ref{slice}.
\item\textit{The composite
\[\mathbb{Z}(X)\xrightarrow{s}\mathbb{Z}(\mathbb{P}(E\oplus O_X))\xrightarrow{t}Th(det(E)^{\vee})(n-1)[2n-2]\xrightarrow{\eta}C\]
is given by the Euler class \(e(E^{\vee})\in H^n(X,W(det(E)))\).}

Since \(s\) factors through the inclusion \(E^{\vee}\subseteq\mathbb{P}(E\oplus O_X)\), we see that \(t\circ s\) is given by the Euler class \(e(E^{\vee})\in\widetilde{CH}^n(X,det(E))\) by Proposition \ref{gysin}. We have a commutative diagram
\[\xymatrixcolsep{1pc}
	\xymatrix
	{
		[\mathbb{Z}(X),Th(det(E)^{\vee})(n-1)[2n-2]]\ar[r]^-{\cong}\ar[d]_{\eta}	&H^n(X,K_n^{MW}(det(E)))\ar[d]_{\eta}\\
		[\mathbb{Z}(X),Th(det(E)^{\vee})(n-2)[2n-3]]\ar[r]^-{\cong}						&H^n(X,K_{n-1}^{MW}(det(E)))
	}
\]
by \cite[Remark 4.2.7, \S 4]{BCDFO}, where the \([-,-]\) is in \(\widetilde{DM}(X)\). Then we use the isomorphisms
\begin{footnotesize}
\[H^n(X,K_{n-1}^{MW}(det(E)))=H^n(X,K_{n-1}^W(det(E)))=H^n(X,W(det(E)))\]
\end{footnotesize}
which follow from the exact sequences
\[0\longrightarrow 2K_{n-1}^M\longrightarrow K_{n-1}^{MW}(det(E))\longrightarrow K_{n-1}^W(det(E))\longrightarrow 0\]
\[0\longrightarrow K_n^W(det(E))\longrightarrow K_{n-1}^W(det(E))\longrightarrow K_{n-1}^M/2\longrightarrow 0.\]
\end{enumerate}
Now we have completed the proof.
\end{enumerate}
\end{proof}
\begin{coro}\label{pbtint}
Let \(S\in Sm/k\), \(X\in Sm/S\) being quasi-projective and \(E\) be a vector bundle of rank \(n\) over \(X\).
\begin{enumerate}
\item We have isomorphisms
\[R(\mathbb{P}(E))/\eta\cong\bigoplus_{i=0}^{n-1}R(X)/\eta(i)[2i]\]
\[Th(E)/\eta\cong R(X)/\eta(n)[2n]\]
in \(\widetilde{DM}(S,R)\).
\item If \(n\) is odd, we have an isomorphism
\[R(\mathbb{P}(E))\cong R(X)\oplus\bigoplus_{i=1}^{\frac{n-1}{2}}R(X)/{\eta}(2i-1)[4i-2]\]
in \(\widetilde{DM}^{eff}(S,R)\).

In particular, we have (\(k=min\{\lfloor\frac{i+1}{2}\rfloor,\frac{n-1}{2}\}\))
\[\widetilde{CH}^{i}(\mathbb{P}(E))=\widetilde{CH}^i(X)\oplus\bigoplus_{j=1}^{k}\eta_{MW}^{i-2j}(X).\]
\item If \(n\) is even, we have a distinguished triangle
\begin{footnotesize}
\[Th(det(E)^{\vee})(n-2)[2n-4]\to R(\mathbb{P}(E))\longrightarrow R(X)\oplus\bigoplus_{i=1}^{\frac{n}{2}-1}R(X)/{\eta}(2i-1)[4i-2]\xrightarrow{\mu}\cdots[1]\]
\end{footnotesize}
in \(\widetilde{DM}(S,R)\) where
\(\mu=0\) if \(e(E^{\vee})=0\in H^n(X,W(det(E)))\). In this case, we have
\[\widetilde{CH}^{i}(\mathbb{P}(E))=\widetilde{CH}^i(X)\oplus\bigoplus_{j=1}^{\frac{n}{2}-1}\eta_{MW}^{i-2j}(X)\oplus\widetilde{CH}^{i-n+1}(X,det(E)^{\vee}).\]
\end{enumerate}
\end{coro}
\begin{proof}
It suffices to do the case \(R=\mathbb{Z}\). Denote by \(q:X\longrightarrow S\) the structure map. We may assume \(E\) is generated by its \(r\) global sections since \(\mathbb{P}(E)=\mathbb{P}(E\otimes L)\) for any \(L\in Pic(X)\) and \(L\) could be chosen ample and equal to a square. Then we obtain a morphism \(f:X\longrightarrow Gr(n,r)\) such that \(E=f^*\mathscr{U}\) where \(\mathscr{U}\) is the tautological bundle. Then applying \(q_{\#}f^*\) on the equations in Theorem \ref{pbt} gives the statements. Note that for any \(L\in Pic(X)\), \(e(E^{\vee})=e(E^{\vee}\otimes L^{\otimes 2})\in H^n(X,W(det(E)))\) by \cite[Theorem 9.1]{L}.
\end{proof}
%
%
\begin{proposition}\label{oriinv}
Suppose \(X\in Sm/k\) and \(\varphi\in GL_n(O_X^{\oplus n})\). Then
\[\mathbb{P}(\varphi):X\times\mathbb{P}^{n-1}\longrightarrow X\times\mathbb{P}^{n-1}\]
is the identity in \(\widetilde{DM}^{eff}(pt)\) if one of the following conditions holds:
\begin{enumerate}
\item The matrix \(\varphi\) is elementary (see \cite[Definition 3]{An});
\item The \(n\) is odd and \(_2CH^*(X)=0\) ;
\item The \(n\) is even and we have \(\varphi=\left(\begin{array}{cc}g^2&\\&Id\end{array}\right)\) where \(g:X\longrightarrow\mathbb{G}_m\) is an invertible function.
\end{enumerate}
\end{proposition}
\begin{proof}
\begin{enumerate}
\item By \cite[Lemma 1]{An}.
\item By Theorem \ref{splitting1} and Theorem \ref{eta} since \(O(1)\) is preserved under \(\mathbb{P}(\varphi)\).
\item The \(\mathbb{P}(\varphi)\) is equal to the composite
\[X\times\mathbb{P}^{n-1}\xrightarrow{\Gamma_{g}}X\times\mathbb{G}_m\times\mathbb{P}^{n-1}\xrightarrow{-^2}X\times\mathbb{G}_m\times\mathbb{P}^{n-1}\xrightarrow{Id_{X}\times u}X\times\mathbb{P}^{n-1}\]
where \(\Gamma_{g}\) is the graph of \(g\) and \(u\) is the morphism
\[\begin{array}{ccccc}\mathbb{G}_m&\times&\mathbb{P}^{n-1}&\longrightarrow&\mathbb{P}^{n-1}\\t&,&(x_0:\cdots:x_{n-1})&\longmapsto&(tx_0:x_1:\cdots:x_{n-1})\end{array}.\]
It is determined by an element in \([\mathbb{Z}(n)[2n-1],\mathbb{Z}(n-1)[2n-2]]_{MW}=W(k)\) by Proposition \ref{ortho}. On the other hand, consider the squaring map \(-^2\) on \(\mathbb{G}_m\), which is determined by an element \(s\in[\mathbb{Z}(1)[1],\mathbb{Z}(1)[1]]_{MW}=GW(k)\). Its image in \(\mathbb{Z}\) is \(2\) since the \(-^2\) induces the twice map on \(O^*(\mathbb{G}_m)/k^*=\mathbb{Z}\). Its image in \(W(k)\) is zero since the fiber of \(1\in\mathbb{G}_m\) under \(-^2\) is the principle divisor of \(<t^2-1>\), where \(t\) is the parameter of \(\mathbb{G}_m\). So we see that \(s\) is the hyperbolic quadratic form \(h\). Hence the composite
\[X\times\mathbb{G}_m\times\mathbb{P}^{n-1}\xrightarrow{-^2}X\times\mathbb{G}_m\times\mathbb{P}^{n-1}\xrightarrow{Id_{X}\times u}X\times\mathbb{P}^{n-1}\]
is equal to the projection map in \(\widetilde{DM}^{eff}(pt)\), which implies that \(\mathbb{P}(\varphi)=Id\).
\end{enumerate}
\end{proof}
\begin{definition}\label{slc}
Suppose \(E\) is a vector bundle over \(X\in Sm/k\) and \(\varphi\in Aut_{O_X}(E)\). We say that \(\varphi\) is \(SL^c\) if \(det(\varphi)\in Aut_{O_X}(det(E))=O_X(X)^*\) lies in \((O_X(X)^*)^2\).
\end{definition}

Suppose \(X\in Sm/k\), \(\varphi\in Aut_{O_X}(O_X^{\oplus n})\) and \(n\) is even. If \(\varphi\) is not \(SL^c\), we couldn't expect \(\mathbb{P}(\varphi)=Id\). An easy example is when \(X=pt\), \(n=2\) and \(\varphi=\left(\begin{array}{cc}g&\\&1\end{array}\right)\), the \(\mathbb{P}(\varphi)\) induces a multiplication by \(<g>\) on \(\widetilde{CH}^*(\mathbb{P}^1)\) by \cite[Lemma 6.3.4]{Mo1}. If \(\varphi\in SL^c(O_X^{\oplus n})\), the \(\varphi\) could be locally written as a composite of elementary matrices and \(\left(\begin{array}{cc}g^2&\\&Id\end{array}\right)\) where \(g\) is invertible.
\begin{theorem}\label{oriinv2}
Let \(S\in Sm/k\) and \(X\in Sm/S\) be quasi-projective with \(_2CH^*(X)=0\), \(E\) be a vector bundle of rank \(n\) over \(X\) and \(\varphi\in Aut_{O_X}(E)\).
\begin{enumerate}
\item If \(n\) is odd, the morphism \(\mathbb{P}(\varphi)=Id\) in \(\widetilde{DM}^{eff}(S,R)\).
\item If \(n\geq dim(X)+2\) is even and \(e(E^{\vee})=0\in H^n(X,W(det(E)))\), the morphism \(\mathbb{P}(\varphi)=Id\) in \(\widetilde{DM}(S,R)\) if \(\varphi\) is \(SL^c\).
\end{enumerate}
\end{theorem}
\begin{proof}
We only prove the second statement, the first one follows by essentially the same method. We may as well assume \(S=X\). We have an isomorphism
\[\mathbb{Z}(\mathbb{P}(E))=\mathbb{Z}(X)\oplus\bigoplus_{i=1}^{\frac{n}{2}-1}\mathbb{Z}(X)/\eta(2i-1)[4i-2]\oplus Th(det(E)^{\vee})(n-2)[2n-4]\]
in \(\widetilde{DM}(X)\) by Corollary \ref{pbtint}. By Theorem \ref{eta}, we only have to prove \(\mathbb{P}(\varphi)^*\) acts trivially on the group
\[Hom_{\widetilde{DM}(X)}(\mathbb{Z}(\mathbb{P}(E)),Th(det(E)^{\vee})(n-2)[2n-4]).\]
Suppose \(f\) belongs to the group above. The \(f\) (resp. \(\mathbb{P}(\varphi)^*(f)\)) can be written as \((a,\xi_i,u)\) (resp. \((a',\xi'_i,u')\)) according to the decomposition of \(\mathbb{Z}(\mathbb{P}(E))\). We have
\[Hom_{\widetilde{DM}(X)}(\mathbb{Z}(X)/\eta(2i-1)[4i-2],Th(det(E)^{\vee})(n-2)[2n-4])=\eta_{MW}^{n-2i-1}(X)\]
by Corollary \ref{pbtint} and strong duality of \(\mathbb{Z}/\eta\), so \(\xi_i=\xi'_i\) by Theorem \ref{eta}. Moreover, since \(n\geq dim(X)+2\), we have
\[Hom_{\widetilde{DM}(X)}(\mathbb{Z}(X),Th(det(E)^{\vee})(n-2)[2n-4])=\widetilde{CH}^{n-1}(X,det(E))=0\]
by \cite[Remark 4.2.7, \S 3]{BCDFO}, hence \(a=a'=0\). Finally, we have \(u=u'\) locally holds by Proposition \ref{oriinv} and
\[End_{\widetilde{DM}(X)}(Th(det(E)^{\vee})(n-2)[2n-4])=GW(X)\]
is a sheaf with respect to \(X\). Hence \(u=u'\). So we have completed the proof.
\end{proof}

If \(n<dim(X)+2\) and \(n\) is even, it might be true that \(\mathbb{P}(\varphi)\neq Id\) even if \(\varphi\) is \(SL^c\). Let us give an example. Suppose \(X=SL_2=\{(a,b,c,d)\in\mathbb{A}^4|ad-bc=1\}\). Then there is a canonical automorphism \(\varphi\) of \(O_X^{\oplus 2}\) defined by the action of \(SL_2\) on \(k^{\oplus 2}\). There is an isomorphism
\[\widetilde{CH}^1(X\times\mathbb{P}^1)=\widetilde{CH}^1(X)\oplus GW(X).\]
We have a composite
\[f:X\xrightarrow{(id,(1:0))}X\times\mathbb{P}^{1}\xrightarrow{\mathbb{P}(\varphi)}X\times\mathbb{P}^{1}\xrightarrow{p}\mathbb{P}^{1}\]
where \(s:=p^*({\tau}_2)=(0,1)\), \(\mathbb{P}(\varphi)^*(s)=(t,1)\) and \(f^*({\tau}_2)=t\). But in our setting \(f\) is equal to the composite
\[\begin{array}{ccccc}SL_2&\longrightarrow&\mathbb{A}^2\setminus 0&\xrightarrow{Hopf}&\mathbb{P}^1\\\left(\begin{array}{cc}a&b\\c&d\end{array}\right)&\longmapsto&(a,c)&&\end{array}\]
where the first arrow is an \(\mathbb{A}^1\)-bundle. So
\[f^*:\widetilde{CH}^1(\mathbb{P}^1)\longrightarrow\widetilde{CH}^1(SL_2)\]
can be identified with the map \(GW(k)\longrightarrow W(k)\). Hence \(t\neq 0\).

Now we compute the MW-motives of blow-ups as an application of the results above. Suppose \(S\in Sm/k\), \(Z, X\in Sm/S\) with \(Z\) being closed in \(X\), \(X'=Bl_Z(X)\) and \(Z'=X'\times_XZ\cong\mathbb{P}(N_{Z/X})\). Thus we have a Cartesian square
\[
	\xymatrix
	{
		Z'\ar[r]^{i'}\ar[d]_{f_Z}	&X'\ar[d]_f\\
		Z\ar[r]^{i}						&X
	}.
\]
\begin{proposition}\label{blowuptri}
We have a distinguished triangle
\[\mathbb{Z}(Z')\xrightarrow{(i,f_Z)}\mathbb{Z}(X')\oplus\mathbb{Z}(Z)\xrightarrow{f-i}\mathbb{Z}(X)\longrightarrow\mathbb{Z}(Z')[1]\]
in \(\widetilde{DM}^{eff}(S)\).
\end{proposition}
\begin{proof}
See \cite[Section 3, Remark 2.30, page 118]{MV} by using Remark \ref{shadj}.
\end{proof}
\begin{theorem}\label{blowup}
Suppose \(n:=codim_X(Z)\) is odd and \(Z\) is quasi-projective. We have
\[R(X')\cong R(X)\oplus\bigoplus_{i=1}^{\frac{n-1}{2}}R(Z)/{\eta}(2i-1)[4i-2]\]
in \(\widetilde{DM}^{eff}(S,R)\).

In particular, we have (\(k=min\{\lfloor\frac{i+1}{2}\rfloor,\frac{n-1}{2}\}\))
\[\widetilde{CH}^{i}(X')=\widetilde{CH}^i(X)\oplus\bigoplus_{j=1}^{k}\eta_{MW}^{i-2j}(Z).\]
\end{theorem}
\begin{proof}
It suffices to do the case \(R=\mathbb{Z}\). By Corollary \ref{pbtint}, \(f_Z\) splits hence the triangle in Proposition \ref{blowuptri} induces a distinguished triangle
\[\bigoplus_{i=1}^{\frac{n-1}{2}}\mathbb{Z}(Z)/{\eta}(2i-1)[4i-2]\longrightarrow\mathbb{Z}(X')\xrightarrow{f}\mathbb{Z}(X)\longrightarrow\cdots[1],\]
where the splitting of the former implies the splitting of the latter.
Denote by \(X''=Bl_{Z\times 0}(X\times\mathbb{A}^2)\) and \(Z''=(Z\times 0)\times_{X\times\mathbb{A}^2}X''\). We have a morphism between distinguished triangles
\[
	\xymatrix
	{
		\mathbb{Z}(Z')\ar[r]\ar[d]_{\alpha}	&\mathbb{Z}(X')\oplus\mathbb{Z}(Z)\ar[r]\ar[d]_{\gamma}	&\mathbb{Z}(X)\ar[r]\ar[d]_{\beta}^{\cong}	&\\
		\mathbb{Z}(Z'')\ar[r]^-v					&\mathbb{Z}(X'')\oplus\mathbb{Z}(Z\times 0)\ar[r]^-u	&\mathbb{Z}(X\times\mathbb{A}^2)\ar[r]				&
	}
\]
where \(\beta\) is the zero section and \(\alpha\) has a section \(p\) by Corollary \ref{pbtint}. Now \(\beta\) factors through \(\mathbb{Z}(X'')\) by construction (by using the section constantly equal to \(1\), which is homotopic to \(\beta\)). Hence the lower triangle splits. Denote by \(w\) a section of \(v\), then \(p\circ w\circ\gamma\) splits the upper triangle. The second statement is clear.
\end{proof}
{}
\Addresses
\end{document}